\documentclass{amsart}
\usepackage[english]{babel}
\usepackage[T1]{fontenc}
\usepackage{amsmath}
\usepackage{amsthm}
\usepackage{amssymb}
\usepackage{graphicx}
\usepackage{amscd}

\usepackage{a4wide}

\usepackage{color,comment}

\newcommand{\E}{\mathbb{E}}
\newcommand{\J}{\mathbb{I}}
\newcommand{\CB}{\mathcal{B}}

\newcommand{\CY}{\mathcal{Y}}

\def\CL { \mathcal{L}}

\def\CB {\mathcal{B}}

\def\CL { \mathcal{L}}

\def\CN {\mathcal{N}}

\def\b {\beta}
\def\a {\alpha}

\def \th{\theta}

\def\eps {\epsilon}

\def\a {\alpha}

\def\ra {\zeta}

\def\J {\mathbb{I}}

\def\U {\mathbb{U}}

\def\N {\mathbb{N}}

\def\E {\mathbb{E}}
\def\RE {\mathbb{R}}

\def\T {\mathbb{T}}

\newcommand{\fer}[1]{(\ref{#1})}

\def \la {\E [ }
\def \ra {] }
\def \laa {\E [ }
\def \raa {] }

\newtheorem{definition}{Definition}
\newtheorem{theorem}[definition]{Theorem}
\newtheorem{lemma}[definition]{Lemma}
\newtheorem{proposition}[definition]{Proposition}

\usepackage{verbatim}

\begin{document}

\title{Mean field dynamics of collisional processes with duplication, loss and copy }

\author{ Federico Bassetti, Giuseppe Toscani}

\maketitle

\begin{abstract}
In this paper we introduce and discuss kinetic equations for the evolution of the
probability distribution of the number of particles in a population subject to binary
interactions. The microscopic binary law of interaction is assumed to be dependent on
fixed-in-time random parameters which describe both birth and death of particles, and
the migration rule.  These assumptions lead to a Boltzmann-type equation that in the
case in which the mean number of the population is preserved, can be fully studied, by
obtaining in some case  the analytic description of the steady profile. In all cases,
however, a simpler kinetic description can be derived, by considering the limit of
quasi-invariant interactions.  This procedure allows to describe the evolution process
in terms of a linear kinetic transport-type equation. Among the various processes that
can be described in this way, one recognizes the Lea-Coulson model of mutation
processes in bacteria, a variation of the original model proposed by Luria and
Delbr\"uck.
\end{abstract}

\section{Introduction}

The description of emerging  collective behaviors and self-organization in multi-agent
interactions started to gain popularity in the recent years and it represents one of
the major challenges in contemporary mathematical modeling. In the biological context,
the emergent behavior of bird flocks, fish schools or bacteria aggregations, among
others, is a major research topic in population and behavioral biology and ecology
\cite{CFRT, CFTV, CM08, CS07, Cuc07, HT08}. Other important examples of emergent behaviors
describe building of tumors by cancer cells and their migration through the tissues
\cite{Bel08, Bel04, Bell08, Maini08, KP12}.  Another famous example to consider in this contest
is the classical Luria--Delbr\"uck mutation problem \cite{LC49, LD42, LD43}. Nonlinear
statistical physics represents a powerful tool to describe these different biological
phenomena. In particular,  methods borrowed from kinetic theory of rarefied
gases have been successfully employed to construct master equations of Boltzmann type,
usually referred to as kinetic equations, describing the time-evolution of the number
density of the population and, eventually, the emergence of universal behaviors
through their equilibria \cite{NPT,PT13}.

The building block of kinetic theory is represented by binary interactions, which,
similarly to binary interactions between particles velocities in the classical kinetic
theory of rarefied gases, describe the variation law  of some selected agent
characteristic, like its number. Then, the microscopic law of variation of the number
density consequent to the (fixed-in-time) way of interaction, is able to capture both
the time evolution and the steady profile, in presence of some conservation law
\cite{NPT,PT13}.

In this paper, we are interested in studying processes in which a huge population of
$N$ interacting agents can be characterized in terms of some scalar quantity assuming
integer non-negative values, say $(V_1,\dots,V_N)$.  The binary interactions between
agents are described by the following rule. When two agents $i$ and $j$ interact,
their pre-interaction values $(V_i,V_j)$ change to
\begin{equation}\label{coll0}
 V_i'= \sum_{k=1}^{V_i}X_{ik} + \sum_{k=1}^{V_j} Y_{ik}, \qquad
 V_j'= \sum_{k=1}^{V_j}X_{jk}  + \sum_{k=1}^{V_i} Y_{jk}.
\end{equation}
In \fer{coll0}, the quantities  $X_{ik},Y_{ik},X_{jk},Y_{jk}$  are independent random
variables which assume non-negative integer values,
the $X$'s with density function $p_X$ and
the $Y$'s with density function $p_Y$.

In \cite{Grillietal}, (Case $1$), a collisional process of type \eqref{coll0}, with
$X_{jk} \in \{0,1,2\}$ and $Y_{jk}\in \{0,1\}$, has been introduced to describe the
evolution of the gene-family abundance (the number of genes of a given family found in
a genome) trough a minimal dynamics of {\it duplication, loss} and {\it interspecies horizontal
gene transfer} (HGT) According to this model, a fixed number $N$ of species-genomes
interact pairwise randomly, so that a given family can gain genes by interactions
associated with HGT events. When two species interact they can exchange genes by HGT
by drawing them from each other with Bernoulli trials of probability $p_h$. In the
same time, they draw from their own genome genes to be lost (with probability $p_l$)
and duplicated (with probability $p_d$). In summary, in this case $P\{X_{ik}=0\}=p_l$,
$P\{X_{ik}=2\}=p_d$, $P\{X_{ik}=1\}=1-p_l-p_d$, $P\{Y_{ik}=1\}=1-P\{Y_{ik}=0\}=p_h$.
Note that, by taking $p_d+p_h=p_l$, the mean number of genes in a given family
is conserved in each collision.

More generally, binary interactions  \eqref{coll0} can be seen as a process in which the $N$
agents can vary their given quantity   $V_i$ of some objects (the genes in Case 1)
according to a collisional processes with {\it duplication, loss and copy}.
An alternative interpretation of the process relies in considering populations in
place of individuals,  and individuals in place of objects. In this case, populations
evolve following a classical branching process, and in addition interact with the
other branching populations and exchange ``migrants''. Note that, within this
interpretation, the $k$-th individual of the population $i$, interacting with
population $j$,  has $X_{ik}$ children  that remain in the population $i$ (do not
migrate) and $Y_{jk}$ that migrate to population $j$.

An example of this second interpretation (Case $2$) is a process in which
$N$ populations of cells interact and are subject to a mutation process.
Here $V_j$ represent the number of mutant cells in the $j$ population.
Each mutating cell can produce a clone in the same population
with probability $p$ and, independently, it can produce a clone which ``migrates'' to the other population
with probability $q$. In other words  $P\{X_{ik}=2\}=p$, and
 $P\{X_{ik}=1\}=1-p$ and $P\{Y_{ik}=1\}=1-P\{Y_{ik}=0\}=q$.
Here typically one consider situation in which the mean number of mutants
grown in time.
As we shall see this process of growth is similar to the linear process for cells
mutation proposed by Lea and Coulson  \cite{LC49}, which has its origin in a series of
pioneering experiments proposed by Luria and Delbr\"uck \cite{LD43}.

Our aim here is to make use of classical methods of kinetic theory to provide a
kinetic description of the evolution in time of a multi-agent system obeying to binary
interactions of type \fer{coll0}. Making use of this collisional mechanism between
individuals, we introduce a bilinear Boltzmann-type equation which describes the
behavior of the population in terms of its density $f(v,t)$, where $v$ represents the
number of objects.  Next, in the asymptotic procedure usually referred to as
\emph{grazing collision} limit, we obtain a simpler linear equation in divergence
form, which retains many properties of the underlying Boltzmann equation, and in
addition can be studied in detail.

Having in mind as prototypes both the horizontal gene transfer model in
\cite{Grillietal} (Case $1$), and the variant of  Luria--Delbr\"uck mutation model in
\cite{KP12, Tos13} (Case $2$), we will split our analysis in two sub-cases, identified
by the evolution of the mean density. Thus, we will limit our study to the cases in
which the mean density of the population remains constant, or grows in time. In both
cases, we can identify in a precise way the large time behavior of the solution. In
particular, in presence of conservation of the mean number of objects (Case $1$), it
will be shown that the solution density converges to a steady state profile, which
depends heavily from the microscopic interactions.

In more details, in Section \ref{sec:kin} we will briefly introduce the kinetic
description of our collisional models, coupling them with some direct physical
consequences. This allows the interested reader to take an exhaustive view of the
modeling assumptions, by comparing both the nonlinear Boltzmann-type equation with its
linear asymptotics, named quasi-invariant limit. Also, the main differences between
Cases $1$ and $2$ are here underlined. Next, we collect in Section \ref{sec:boltz} the
main results on existence and uniqueness of solutions to the mathematical models, by
computing additionally some of the relevant mean quantities. Here, convergence to
equilibrium for collisions of type \fer{coll0}, which imply conservation of the mean,
value is studied in full details. Section \ref{sec:scaled} will deal with the
large-time behaviour of the solution in the case in which collisions of type
\fer{coll0} imply the growth of the mean value. In the case of growth, by suitably
scaling the solution with respect to its time-dependent mean value we will show that
the scaled solution converges towards a fixed steady profile as time tends towards
infinity. Last, Section \ref{Sec:grazing} will describe and justify from a
mathematical point of view the quasi-invariant asymptotics procedure leading from the
Boltzmann-type equation to the linear one. Some technical results are for the sake of
readability postponed to the Appendix.

\section{The kinetic equation and its quasi-invariant collision limit}\label{sec:kin}

\subsection{The kinetic equation}
Performing the usual  mean-field approximation when the number of particles (colliding
entities) goes to infinity, the evolution in time of the number density can be
quantitatively described by a bilinear Boltzmann-type equation. In this equation, the
time variation of the density $f_t(v)=f(v,t)$, with $v \in \N$ and $t > 0$, follows
from  a balance between gain and loss terms, that, for the given number $v$, take into
account all the interactions of type \fer{coll0} which end up with the number $v$
(gain term) as well as all the interactions which, starting from the number $v$, lose
this value after interaction (loss term). This Boltzmann-type equation reads
\cite{PT13}
\begin{equation}\label{boltz}
\frac{\partial}{\partial t} f_t (v)= Q^+(f_t,f_t)(v) -f_t(v).
\end{equation}
In \fer{boltz},  $Q^+$ is the gain collision operator defined, for any pair of densities $f$ and $g$ on $\N$, by
\begin{equation}\label{eq.Q+}
Q^+(f,g)(v)= \text{Prob}\Big (  \sum_{i=1}^{V_1} Y_i + \sum_{i=1}^{V_2} X_i =v\Big ),
\end{equation}
where $V_1,V_2,X_1,X_2,\dots,Y_1,Y_2\dots$ are stochastically independent,
$V_1$ has density $f$, $V_2$ has density $g$, the $X_i$'s have the same law of a
random variable $X$ with  density $p_X$ and the $Y_i$'s have the same law of a random
variable $Y$ with density $p_Y$.

Equation \fer{boltz} is coupled with an initial condition $f_0(v)$, which is here assumed to be a probability density on $\N$.
Considering that the variable $v$ can only assume values in $\N$,
equation \eqref{boltz} can be fruitfully rewritten in terms of  probability generating functions (p.g.f.)
\[
\hat f_t(z):=\sum_{v \geq 0} z^v  f_t(v) \qquad z \in [0,1].
\]
A standard computation shows that $\hat f_t(z)$ satisfies the (simpler) nonlinear equation
\begin{equation}\label{boltz-genfun}
\begin{split}
& \frac{\partial }{\partial t} \hat f_t(z)=\hat f_t(\hat p_X(z)) \hat f_t(\hat p_Y(z)) -\hat f_t(z) \qquad z \in [0,1], \,\, t>0, \\
& \hat f_0(z)=\sum_{v \geq 0} z^{v}  f_0(v) \\
\end{split}
\end{equation}
where
\[
 \hat p_X(z)= \la z^{X}  \ra=\sum_{m \geq 0} z^m  p_X(m) \quad
\text{and} \quad
\hat p_Y(z)=\la z^{Y}  \ra=\sum_{m \geq 1} z^m p_Y(m).
\]
Equation \fer{boltz-genfun} is the analogous of the Fourier transformed Boltzmann equation for maxwell pseudo-molecules, introduced in kinetic theory of rarefied gases by Bobylev \cite{Bob}.

To avoid trivial situations, in what follows we shall assume that both
$P\{X=0\}\not=1$ and $P\{Y=0\}\not=1$.

Existence and uniqueness of solution of equation \eqref{boltz-genfun} can be proven in a standard way (cf. Section \ref{Sec:3} for details).

In this paper, we will only consider random variables $X$ and $Y$ such that
\begin{equation}\label{condizionemediafinita}
 \la X^r +Y^r\ra <+\infty
\end{equation}
for some $r \geq 1$. In presence of condition \eqref{condizionemediafinita} for some
integer $r \geq 1$, and assuming that the initial density $f_0(v)$ has bounded moments
up to order $r$, one can easily reckon explicit expressions for the
evolution of the moments of $f_t(v)$. In particular, the  mean $M_1(f_t)=\sum_{v} v
f_t(v)$ of $f_t$ evolves according to
\begin{equation}\label{evmeanA}
M_1(f_t)=M_1(f_0) e^{\alpha_1 t},
\end{equation}
where $\alpha_1:=\E[(X+Y)]-1$. Indeed, recall that
for any p.g.f.   $\hat \rho$ (of a probability $\rho$ on  $\N$)
\[
M_1(\rho):=\sum_v  v \rho(v)= \partial_z \hat \rho(z)|_{z=1},
\]
where $\partial_z \hat \rho(z)|_{z=1}:=\lim_{z \to 1^-}\partial_z \hat \rho(z) <+\infty$ if and only if $\sum_v  v \rho(v)<+\infty$.
By taking the derivative with respect to $z$ on both sides of \eqref{boltz-genfun}, and evaluating the resulting equation in $z=1$, one obtains
\[
\frac{d}{dt} M_1(f_t)=\a_1 M_1(f_t),
\]
which gives \eqref{evmeanA}. Analogous computations can be done to evaluate higher-order cumulant functions of the density $f_t$.
An explicit expression of the variance of $f_t$ is presented in Proposition \ref{propVar}.

Note that the condition
\begin{equation}\label{condizionemedia}
\la X +Y\ra =1
\end{equation}
implies that the mean remains constant  in time, i.e. $M_1(f_t)=m_0$  for every $t>0$, as in Case 1.

\subsection{The quasi-invariant collision limit}\label{S:Quasiinvariant}

The bilinear Boltzmann-like equation \fer{boltz}, fruitfully written in the form
\fer{boltz-genfun}, is the starting point to obtain, in a well-established asymptotic
procedure, simpler models that are reminiscent of the binary collision rules, and,
while maintaining most of the properties of the nonlinear kinetic model, result to be
linear. This kind of asymptotic procedure is close to the so-called \emph{grazing
collision} limit for the Boltzmann equation \cite{Villani98}, and has been widely used
in kinetic theory to obtain Fokker-Planck type equations to describe,  among others,
the cooling in granular gases \cite{Furioli2012, Pareschi06}, wealth distribution in a
multi-agent society \cite{CPT}, price formation \cite{BCMW2, BCMW1}, and opinion formation  \cite{DMPW, Tos06}.

This asymptotics procedure is based on the following assumptions. Given a small
positive parameter $\eps$, assume that the random variables $X$ and $Y$ are defined by
\begin{equation}\label{variablegrazing}
 X{=}\eta_1 \tilde X + (1-\eta_1) \qquad \text{and} \qquad Y{=}\eta_2 \tilde Y,
\end{equation}
where   $\tilde X,\tilde Y,\eta_1,\eta_2$ are independent random variables, and, for
some constants $b_1,b_2>0$ such that $b_i\eps \le 1$, $i =1,2$, it holds
 \[
 P\{\eta_i=1\}=1-P\{\eta_i=0\}=b_i \eps, \qquad i=1,2. %
 \]
Note that by construction the laws of $X$ and $Y$ depend on $\eps$, so that $X=X_\eps$
and $Y=Y_\eps$. Nevertheless, for the shake of notational simplicity, the dependence
in $\eps$ has been omitted. Since $\eps$ is assumed to be small, the post-collision
quantities $(V'_1, V'_2)$ remain equal to the pre-collision ones $(V_1,V_2)$ with high
probability. Indeed, with high probability, $X_\eps$ is  equal to one and $Y_\eps$ to
zero. In order to evaluate the mean value of $f_t(v)$ in correspondence to the collision defined
by \fer{variablegrazing}, observe that
\[
\E[X_\eps+ Y_\eps]-1=\eps[b_1(\E[\tilde X]-1)+b_2\E[\tilde Y]]=:\eps \bar
\alpha_1.
\]
Hence,
 \eqref{evmeanA} takes the form
 \begin{equation}\label{new-mean}
M_1(f_t)= \partial_z \hat f_t(1)=m_0 e^{\bar\alpha_1 \eps t}.
 \end{equation}
This shows that, in presence of the small parameter $\eps$, to observe the same variation of
the mean value corresponding to $\eps=1$, it is necessary to wait a longer time
$\tau$, given by $\tau = t/\eps$.

By virtue of \fer{variablegrazing} one obtains
\[
\hat p_X(z)=z+ \eps b_1( \hat p_{\tilde X}(z) - z)  \quad \text{and} \quad \hat
p_Y(z)=1+ \eps b_2  (\hat p_{\tilde Y}(z)-1).
\]
Therefore, if  $\hat f_{t}(z)=\hat f_{t,\eps}(z)$ is a solution of \eqref{boltz-genfun} corresponding to random variables  $X$ and $Y$ defined in  \eqref{variablegrazing},
and the initial density $f_0$ has finite mean $m_0$, by expanding
 $\hat f_t(\hat p_X(z))$  and $\hat f_t(\hat p_Y(z))$ in Taylor's series around  $z$ (around $1$, respectively) gives
\[
 \hat f_t(\hat p_X(z))=\hat f_t(z) +\eps b_1 ( \hat p_{\tilde X}(z) - z)  \partial_z \hat f_t(z) +\eps b_1  R_{1,\eps}(t,z),
\]
and
\[
 \hat f_t(\hat p_Y(z))=\hat f_t(1)+\eps b_2 (\hat p_{\tilde Y}(z)-1) \partial_z \hat f_t(1) +\eps b_2 R_{2,\eps}(t,z),
\]
where $R_{1,\eps}$ and $R_{2,\eps}$ denote the remainders. Since by definition $\hat
f_t(1)=1$, recalling also \eqref{new-mean}, the collision term in \eqref{boltz-genfun} takes the form
 \[
\hat f_t (\hat p_X(z) )  \hat f_t(\hat p_Y(z)) -\hat f_t(z) = \eps \left[ b_2 \hat
f_t(z)(\hat p_{\tilde Y}(z)-1) M_1(f_t)  +  b_1 ( \hat p_{\tilde X}(z) - z)
\partial_z \hat f_t(z)\right] + \eps R_{3,\eps}(t,z),
 \]
where $R_{3,\eps}(t,z)$ denotes the new remainder term. In view of the previous remark on the evolution in time of the mean value, let us set
$ \tau = t/\eps$ and  $\hat g_{\tau,\eps}(z)=\hat f_{\tau/\eps}(z)$. Then, $\hat
g_{\tau,\eps}(z)$ satisfies the  equation
\begin{equation}\label{eqrescaled}
 \partial_\tau \hat g_{\tau,\eps}(z) =\hat g_{\tau,\eps}(z) b_2(\hat p_{\tilde Y}(z)-1) M_1(g_{\tau,\eps})  +
b_1(\hat p_{\tilde X}(z) - z)   \partial_z \hat g_{\tau,\eps}(z) +  R_{3,\eps}(\tau/\eps,z).
\end{equation}
Let us assume that $R_{3,\eps} \to 0$ as $\eps \to 0$. Since $M_1(g_{\tau,\eps})=m_0
e^{\bar \alpha_1 \tau}$, letting $\eps \to 0$, and defining  $g_\tau(z)=\lim_{\eps
\to 0}\hat g_{\tau,\eps}(z)$ shows that $g_\tau$ satisfies the equation
\begin{equation}\label{equationgrazing2}
\begin{split}
 & \partial_\tau  \hat g_\tau(z) =m_0 e^{\bar \alpha_1 \tau} b_2(\hat p_{\tilde Y}(z)-1) \hat g_\tau(z) + b_1 ( \hat p_{\tilde X}(z) - z) \partial_z \hat g_\tau(z),
\\
& \hat g_0(z)=\hat f_0(z). \\
\end{split}
\end{equation}
While the actual derivation of equation \fer{equationgrazing2} is largely formal,
every step can be made rigorous. We postpone a precise statement of the limit
procedure and its detailed proof to Section \ref{Sec:grazing} (cf. Proposition
\ref{prop:limite}).

Here it is worth noticing that from the previous equation it follows that  the mean of
$g_\tau(\cdot)$ is $m_0  e^{\bar \alpha_1 \tau}$. Indeed recall that
$y(\tau):=\partial_z \hat g_\tau(z)|_{z=1}=M_1(g_\tau(\cdot))$, so that deriving with
respect to $z$ equation \eqref{equationgrazing2} and evaluating  the resulting
equation for $z=1$
 one gets
\[
 \dot y(\tau)=b_1( \E[\tilde X]-1)y(\tau)+b_2 m_0e^{\bar \alpha_1 \tau} \E[\tilde Y].
\]
 This implies
\[
 y(\tau)=m_0e^{\bar \alpha_1 \tau}+ e^{b_1\E[\tilde Y]\tau}(y(0)-m_0),
\]
and the claim follows since the initial condition is $y(0)=m_0$.

\subsection{The Lea-Coulson equation for mutation processes}

Consider  the special case in which $P\{\tilde X =2\}=1-P\{\tilde X =1\}=p$ and
$P\{\tilde Y=1\}=1-P\{\tilde Y=0\}=q$. With this choice
$p_{\tilde X}(z)-z=p z(z-1)$ and $p_{\tilde Y}(z)-1=q(z-1)$.
By setting $\beta_1=pb_1+qb_2$, $\beta_2=b_1p$ and $\mu=b_2q m_0$,   equation \eqref{equationgrazing2} takes the form
\begin{equation}\label{eq:LeaCoulson0}
 \partial_t  \hat g(t,z) = (z-1)\left\{ \beta_2 z \partial_z \hat g(t,z)+ \mu  e^{\beta_1 t} \hat g(t,z) \right\}.
\end{equation}
Equation \eqref{eq:LeaCoulson0} is referred to in the literature as the Luria--Delbr\"uck model.
The model deals with the estimation of mutation rates, and has its origin in a series of classic
experiments pioneered by Luria and Delbr\"uck \cite{LD43}, which were at the basis of the construction of a mathematical model able to estimate them.
The original model proposed by Luria and Delbr\"uck assumed
deterministic growth of mutant cells, which seemed a too stringent assumption to allow for efficient extraction of reliable information
about mutation rates from experimental data. This shortcoming of the
model of Luria and Delbr\"uck was some year later remedied by a
slightly different mathematical formulation proposed by Lea and
Coulson \cite{LC49}, who adopted the Yule stochastic birth process to
mimic the growth of mutant cells. In the ensuing decades the
Lea--Coulson formulation occupied  a so prominent  place in the study
of mutation rates that the Lea--Coulson formulation is now commonly
referred to as the Luria--Delbr\"uck model \cite{Zheng}.

The classical linear procedure of growth of both normal and mutant cells, that give rise to equation \eqref{eq:LeaCoulson0} has been merged into the framework of kinetic theory in \cite{KP12}. There, the mutation problem was described in terms of linear interactions, and the solution to the underlying linear Boltzmann equation was shown to converge towards the solution to the Lea--Coulson model. A precise statement of this convergence result has been subsequently done in \cite{Tos13}.

While equation \fer{eq:LeaCoulson0} coincides with the mathematical formulation of Lea and Coulson \cite{Zheng}, it is important to remark that the meaning of the various (eventually time dependent) coefficients in \fer{eq:LeaCoulson0} assume a different meaning depending on the original model from which this equation comes from.

In the original formulation, equation \fer{eq:LeaCoulson0} is recovered through the following assumptions. The process starts at time $t=0$ with one normal cell and no
mutants.  Normal cells are assumed to grow deterministically at a
constant rate $\b_1$. Therefore the number of normal cells at time
$t>0$  is $N(t)= e^{\b_1 t}$. Mutants grow deterministically at a constant rate $\b_2$. If a
mutant is generated by a normal cell at time $s> 0$, then the clone
spawned by this mutant will be of size $e^{\b_2(t-s)}$ for any
$t>s$.  Mutations occur randomly at a rate proportional to $N(t)$.
If $\mu$ denotes the per-cell per-unit-time mutation rate, then the
standard assumption is that mutations occur in accordance with a
Poisson process having the intensity function
 \begin{equation}\label{me}
\nu(t) = \mu N(t)= \mu e^{\b_1 t}.
 \end{equation}
Consequently, the expected number of mutations occurring in the time
interval $[0,t)$ is
 \[
\int_0^t \nu(s) \, ds = \frac \mu{\b_1}\left( e^{\b_1 t} -1 \right).
 \]
The formulation of Lea-Coulson differs from the previous Luria-Delbr\"uck original formulation
in that mutant cell
growth is described by a stochastic birth process.
If a mutant is generated by a normal cell at time
$s > 0$, then at any given time $t> s$
the size of the clone spawned by that mutant will have the same distribution
as $Y(t-s)$, where $Y(s)$ is   a Yule
process having birth rate $\beta_2$ and satisfying $Y(0)=1$. In this way if
$W(t)$ denotes the number of mutants existing at time $t$, then
 \begin{equation}\label{LL}
W(t) = \sum_{i =1}^{M(t)} Y_i(t -\tau_i), \quad {\rm if}
\quad M(t) \ge 1,
 \end{equation}
while
 \[
W(t) = 0 \quad {\rm if} \quad M(t) = 0\, .
 \]
Here, $\tau_i$ are the random times at which mutations occur, $M(t)$ stands for the
mutation process which is a Poisson process with intensity function $\nu(t)$ given in
\fer{me}, and $(Y_i(s)_{i \geq 1}$ is a sequence of independent Yule processes. In
particular,   the probability generating function of $W(t)$ satisfies
\fer{eq:LeaCoulson0} \cite{Zheng}.

Unlike the previous approach, the Lea--Coulson equation
\fer{eq:LeaCoulson0} follows here by linearizing a  nonlinear equation of Boltzmann type
based on classical binary interactions, and consequently from a nonlinear mechanism,
in which the number of cells is subject only to duplication (the variable $X$) and
mutation (the migration variable $Y$). In particular, the initial condition is here
different, and the coefficients $\beta_1$ and $\beta_2$ are no more independent. In
addition, the exponential in time term in front of the probability generating function
$\hat g$ is in our derivation generated by the evolution of the mean value.
Nevertheless, from a mathematical point of view, this derivation could share a new
light on the large-time behaviour of the solution to \fer{eq:LeaCoulson0}, which is
now connected with the solution of a nonlinear Boltzmann-type equation, that, unlike the linear model, possesses solutions of self-similar type. This new connection will
be studied in more details in a companion paper.

\subsection{Steady states}

When $\E[X+Y]=1$, while $\la X^r +Y^r \ra<+\infty$ for some $r>1$, one can prove that,
for any given constant $m_0$,  the kinetic model \eqref{boltz} possesses a steady
state, namely a density satisfying
 \begin{equation}\label{f-inf}
f_\infty=Q^+(f_\infty,f_\infty),
 \end{equation}
of mean $m_0$. Moreover, any solution $f_t$ with initial density $f_0$ of
mean value $m_0$ converges (pointwise) to the corresponding stationary solution
$f_\infty $ whenever $\sum_v v^r f_0(v)<+\infty$ (cf. Theorem \ref{thmconvergence2}).
In Section \ref{Sec:3} we shall show that $f_t$ converges to $f_\infty$ exponentially
fast in time with respect to suitable Fourier metrics (cf. Theorem
\ref{thmconvergence}).

If  $\E[X^2+Y^2]<+\infty$, one can recover an explicit formula for the second moment
$M_2(f_\infty)$ of $f_\infty$. Indeed by combining Lemma \ref{inequalitysum} in
Appendix with identity \fer{f-inf} one gets
\[
\begin{split}
 M_2(f_\infty) & =M_2(Q^+(f_\infty,f_\infty))=M_2(f_\infty)[\E[X]^2+\E[Y]^2] \\
& \quad +2\E[X]\E[Y]M_1(f_\infty)^2+[Var(X)+Var(Y)]M_1(f_\infty).\\
\end{split}
\]
Since $M_1(f_\infty)=m_0$, it follows that
\[
Var(f_\infty)=\frac{(Var(X)+Var(Y)) m_0}{1-\la X \ra^2- \la Y \ra^2}.
\]
While in principle one can recursively determine every integer moment of the steady
state $f_\infty$,  an explicit general expression for the p.g.f $\hat f_\infty$ is not
available, with the exception of very simple situations. An explicit example is
obtained by fixing $P\{X=1\}=1-P\{X=0\}=p$ and $P\{Y=1\}=1-P\{Y=0\}=q$ with $p+q=1$.
In this case the steady state coincides with a Poisson distribution. To see this,
recall that the p.g.f. of a Poisson distribution of mean $m_0$ is given by $\hat
f_\infty(z)=\exp(m_0(z-1))$. In the special case described above, $\hat p_X(z)= z p +
1-p$ and $\hat p_{Y}(z)=z(1-p)+p$. Thus
\[
\hat Q^+(\hat f_\infty,\hat f_\infty)(z)= \exp(m_0(\hat p_{Y}(z)-1))
\exp(m_0 (\hat p_{Y}(z)-1)) =
\exp(m_0(z-1))=\hat f_\infty(z).
\]
One of the advantages of the quasi-invariant limit approximation is related to the
possibility to obtain an explicit expression for the steady states. Indeed, if
$b_1(\E[\tilde X]-1)+b_2\E[\tilde Y]=0$, the stationary equation reads
\begin{equation}\label{stat_grazing}
  m_0 b_2 (\hat p_{\tilde Y}(z)-1) \hat g_{\infty}(z) + b_1 ( \hat p_{\tilde X}(z) - z) \partial_z \hat g_{\infty}(z)=0
\qquad  z \in  [0,1]
\end{equation}
and gives  the explicit solution
\begin{equation}\label{stat_grazing_sol}
 \hat g_\infty(z)=\exp \Big \{ -m_0 \frac{b_2}{b_1}\int_{z}^1 \frac{(1-\hat p_{\tilde Y}(s))}{(\hat p_{\tilde X}(s) - s)}ds  \Big \}.
\end{equation}
In Proposition \ref{prop:steady_grazing} we will present an interesting probabilistic
interpretation of \eqref{stat_grazing}. Other aspects related to the steady states
will be included there.

\subsection{Steady states for horizontal gene transfer processes}


Horizontal gene transfer (HGT) is believed to be the dominant component of genome
innovation for bacteria. Recent estimates show that at least 32\% of the genes in
prokaryotes have been horizontally transferred \cite{Koonin2001,Koonin2008}. Today the
study of the HGT can be tackled using a growing amount of genomic data. Nevertheless
the study of mathematical models for HGT is still at the beginning. Among the few
mathematical studies on the subject we mention the infinitely many genes model
\cite{BaumdickerA}. This model uses the coalescent theory in order to describe the
underlying phylogenetic tree. A different approach has been considered in
\cite{Grillietal}, which focuses on the description of a collisional mechanism which
describes the time evolution of the gene-family abundance (the number of genes of a
given family found in a genome). As already mentioned in the introduction, in the
model a fixed (very large) number of species genomes interact pairwise randomly. When
two species interact they can exchange genes by HGT by drawing them from each other
with Bernoulli trials of probability $p_h$. In the same time, they draw from their own
genome genes to be lost (with probability $p_l$) and duplicated (with probability
$p_d$). In the model different families are considered subject to a stochastically
independent evolution. Although this simplifying assumption is probably unrealistic,
the effective rates $(p_d,p_h,p_l)$ can vary from family to family, giving rise to the
observed diversity between families. An interesting result, which will be briefly described below, is that the model predicts
an increased dispersion in family abundance as HGT becomes less relevant.

In \cite{Grillietal} it is assumed that each gene can be transferred in a single copy,
that corresponds to assume that $Y$ can take only the values $\{0,1\}$. In what
follows we present a description of the steady states (in the quasi-invariant
approximation) for a slightly more general situation, i.e.  both $\tilde X$ and
$\tilde Y$ take values $0,1,2$. 
For the sake of simplicity
assume that $b_1=b_2=1$ and write
\[
 P\{\tilde X=k\}=p_k, \quad P\{\tilde Y=k\}=q_k \quad k=0,1,2.
\]
Since we need $\E[\tilde X+\tilde Y]=1$, we impose that
\[
 p_1+q_1+2(p_2+q_2)=1.
\]
With these choices it is immediate to conclude that
\[
  \frac{(1-\hat p_{\tilde Y}(s))}{(\hat p_{\tilde X}(s) - s)}=\frac{q_1+q_2 +s q_2}{p_0-sp_2}.
\]
Recalling that (for $\beta\not=0$)
\[
 \int_z^1 \frac{a+bs}{\alpha+\beta s}ds=\frac{b}{\beta}(1-z)+\frac{a\beta-\alpha b}{\beta^2}\log\Big (\frac{\alpha+\beta}{\alpha+z \beta} \Big),
\]
when $p_0,p_2>0$, \eqref{stat_grazing_sol} becomes
\[
 \hat g_\infty(z)=e^{m_0 \frac{q_2}{p_2}(z-1)} \Big ( \frac{1-p_2/p_0}{1-p_2/p_0 z}  \Big )^{\frac{m_0}{p_2^2}((q_1+q_2)p_2+p_0q_2)}.
\]
Note that $\E[\tilde X]<1$ implies both that $p_0>0$ and $p_2<p_0$, so that $1-p_2/p_0>0$. This proves that
\[
\hat g_\infty(y)=\hat p_{\null\,\,\Pi_0}(z)\hat p_{N}(z)
\]
where $p_{\null\,\,\Pi_0}(z)=e^{m_0 \frac{q_2}{p_2}(z-1)}$ is the p.g.f. of a Poisson random variable $\Pi_0$ of mean $m_0q_2/p_2$ and
$\hat p_{N}(z)$ is the p.g.f. of a Negative Binomial random variable of parameter $p=p_2/p_0$ and
$r=\frac{m_0}{p_2^2}((q_1+q_2)p_2+p_0q_2)$. Recall that a Negative binomial random variable of parameter $(p,r)$ has density
\[
 p_N(k)={ k-r-1 \choose k}(1-p)^r p^k \quad k=0,1,\dots
\]
In other words, in this case, the steady state $g_\infty$ is the density of the random variable
\[
V_\infty=\Pi_0+N,
\]
where $\Pi_0$ and $N$ are independent. Since a negative binomial distribution can
be represented as a compound Poisson distribution, i.e.
\[
 N=\sum_{i=1}^{\Pi_1} L_i,
\]
where $\Pi_1$ is a Poisson distribution of mean $r\log(1/(1-p))$ and the $L_i$'s are independent and identically distributed random variables, each one with logarithmic distribution
\[
p_{L_i}(k)=\frac{p^k}{k\log(1/(1-p))} \quad (k=1,2,\dots),
\]
one gets the representation
\[
 V_\infty=\Pi_0+\sum_{i=1}^{\Pi_1} L_i.
\]
If now we set $q_2=0$, we obtain
\[
 \hat g_\infty(z)=\hat p_{N}(z),
\]
where $N$ is a Negative Binomial random variable of parameter $p=p_2/p_0$ and $r=m_0 q_1/p_2$, i.e. of mean $m_0$ and
variance $m_0 p_0/q_1$. Note that this is the model introduced in \cite{Grillietal} and briefly discussed above.
In this case $p_0=p_l$ is the probability of loss of a gene, $p_2=p_d>0$ is the probability of
duplication and $q_1=p_h$ is the probability of HGT. It is interesting to remark that in this case the resulting steady distribution
depends only on the initial mean $m_0$ and on the ratio $p_h/p_d$. Moreover, 
the dispersion measured by the index $Variance/mean=1+p_d/p_h$, increases as the probability $p_h$ of HGT  decreases. 
The limit case is obtained when $p_d=p_2=0$.

Indeed, if $p_2=q_2=0$, direct computations show that
$\hat g_\infty(y)= \exp\{m_0(z-1)\}$, namely that the steady states are given by Poisson distributions.
In other words Poisson distributions are recovered in the case in which  $p_d=0$,  that is when no duplications occur.

Finally, although not relevant for the HGT model, let us consider the case in which $p_2=0$, while $q_2>0$. Here
\[
 \hat g_\infty(y)=\exp\left\{\frac{m_0}{p_0}\left[\frac{q_2}{2}z^2+(q_1+q_2)z -(q_1+\frac{3}{2}q_2) \right]\right\},
\]
or, alternatively
\[
 \hat g_\infty(y)=\exp\left\{\frac{2m_0}{p_0(2q_1+3q_2)}[\hat p_R(z)-1]\right\},
\]
where
\[
 \hat p_R(z)=q_2^*z^2+q_1^*z, \qquad q_2^*:=q_2/(2q_1+3q_2), \, q_1^*=1-q_2^*.
\]
Since $\hat p_R(z)$ is the p.g.f. of a random variable  $R$ taking values $1,2$ with probability $q_1^*,q_2^*$, it follows that
in this case the steady state  $g_\infty$ is the density of the compound Poisson random variable
\[
 V_\infty=\sum_{i=1}^{\Pi_2} R_i,
\]
where $\Pi_2,R_1,R_2,\dots$ are independent, $\Pi_2$ is distributed according to a Poisson distribution of mean $2m_0/{[p_0(2q_1+3q_2)]}$
and the $R_i$s have the same law of $R$.

\section{The kinetic equation: solutions, moments, steady states}\label{sec:boltz}

The remaining of the paper will be devoted to detailed proofs of various results about
existence, uniqueness and asymptotic behavior of both the kinetic Boltzmann-type
equation \fer{boltz-genfun}, and its quasi-invariant limit \fer{equationgrazing2}.

\subsection{Existence and uniqueness of solutions}\label{Sec:3}
The unique solution of \eqref{boltz-genfun} can be written
 in a semi-explicit form by resorting to the so-called Wild series \cite{Wild, PT13}.

\begin{proposition}\label{propWild}
Let $f_0(v)$ be a probability density. Then, the initial value problem for equation
\fer{boltz-genfun}, with initial condition $f_0$,  has a unique global solution. For
any given $t >0$, the unique solution $\hat f_t$ can be written as
\[
 \hat f_t(z)=e^{-t} \sum_{n \geq 0} (1-e^{-t})^{n}  \hat  q_n(z),
\]
where
\(
 \hat  q_0(z)=\hat f_0(z),
\)
and, for any $n \geq 1$,
\[
 \hat  q_n(z):=\frac{1}{n} \sum_{j=0}^{n-1} \hat q_j(\hat p_X(z))  \hat  q_{n-1-j}(\hat p_Y(z)).
\]
Moreover, $\hat f_t$ remains a p.g.f. for any $t > 0$.
\end{proposition}

\begin{proof}
The fact that $\hat f_t(z)$ is a solution can be checked directly.
To prove uniqueness, let  $\hat f_t$ and $\hat g_t$ be two solutions departing from
the same initial value.
Note that in place of  \eqref{boltz} one can consider the same equation in integral form
\begin{equation}\label{boltz-integral}
 f_t(v)-f_0(v)=\int_0^t [ Q^+(f_s,f_s)(v) -f_s(v)  ] ds \qquad v \in \N .
\end{equation}
Define  $u(t):= \sup_{z \in [0,1]} |\hat f_t(z)-\hat g_t(z)| \leq 2$. Using \eqref{boltz-integral} one gets
\(
 u(t) \leq u(0)+ 2 \int_0^t u(s) ds\).
Hence, Gronwall inequality implies
 \(
 u(t) \leq u(0) e^{2t}\).
Since $u(0)=0$ it follows that $u(t)=0$. Finally, by induction  one shows immediately
that $\hat q_n$ is a p.g.f. for every $n\geq 0$. Hence $\hat f_t$ is a p.g.f..
\end{proof}

An immediate consequence of the previous result is that the unique solution to
equation \fer{boltz} can be expressed as
\begin{equation}\label{wild1}
  f_t(v)=e^{-t} \sum_{n \geq 0} (1-e^{-t})^{n}   q_n(v),
\end{equation}
where
\(
   q_0(v)= f_0(v)\)
and, for any $n \geq 1$,
\begin{equation}\label{wild2}
  q_n(v):=\frac{1}{n} \sum_{j=0}^{n-1}  Q^+(q_j,q_{n-1-j})(v).
\end{equation}

\subsection{Steady states and convergence to equilibrium}
For any $r>0$, let us consider the following metric between probability distributions
over non-negative integers
\[
 d_{r}(f,g):=\sup_{z \in (0,1)} \frac{|\hat f(z)-\hat g(z)|}{|1-z|^{r}}.
\]
This is a natural adaptation of the well-known Fourier distance introduced in
\cite{Toscanimetrica}. It is clear that convergence with respect to the metric $d_r$
yields convergence of p.g.f.'s and, hence, point-wise convergence of densities.

In dealing with general  probability measures $f$ and $g$ on  $[0,+\infty)$, it is
also useful to introduce the weighted Laplace transform metric
\[
  d^*_{r}(f,g):=\sup_{\xi > 0 } \frac{|\tilde f(\xi)-\tilde g(\xi)|}{|\xi|^{r}}
\]
where
\[
 \tilde f(\xi)=\int e^{-\xi v} f(dv) \qquad \text{and} \qquad \tilde g(\xi)=\int e^{-\xi v} g(dv).
\]
It is easy to see that  $d_r$ and $d^*_r$ are (topologically) equivalent (cf. Lemma
\ref{lemma3}).

We remark that, in terms of the Laplace transform
\[
 \tilde f_t(\xi)=\sum_v e^{-\xi v} f_t(v),
\]
the Boltzmann-type equation \fer{boltz}  takes the form
\begin{equation}\label{boltz-laplace}
 \frac{\partial }{\partial t} \tilde f_t(\xi)=\tilde f_t(-k_X(\xi)) \tilde f_t(-k_Y(\xi)) -\tilde f_t(\xi), \qquad \xi> 0. 
\end{equation}
Now, $\tilde f_0(\xi)=\sum_v e^{-\xi v}  f_0(v)$ is the Laplace transform of the
initial density, and
\[
 k_X(\xi)=\log(\la e^{-\xi X}\ra) \quad \text{and} \quad k_Y(\xi)=\log(\la e^{-\xi Y}\ra)
\]
are the cumulant functions of $X$ and $Y$.

Given a density $f$ on $\N$ (or more generally a probability measure $f$ on $\RE^+$),
$M_r(f)$ will denote the $r$-moment of $f$, expressed by $\sum_v v^r f(v)$  (or more
generally by $\int v^r f(dv)$). To  denote the variance of a random variable $Z$ we
will write $Var(Z)$. Moreover, if  $f$ is a density (or a probability measure)
$Var(f)$ will denote the variance of a random variable $Z$ with law $f$, that is
$Var(f)=M_2(f)-M_1(f)^2$.
Finally, for $r \geq 1$, let us set
\begin{equation}\label{alpha}
\Delta_r:= \la X \ra^r+ \la Y \ra^r
\quad \text{and} \quad
 \alpha_r=\la X \ra^r+ \la Y \ra^r-1.
\end{equation}

\begin{theorem}[\bf{Contraction in $d_r$}]\label{thmconvergence}
Assume that $\la X^r +Y^r \ra <+\infty$ for some $r \in (1,2]$.
Let
 $f_t^{(1)}$ and $f_t^{(2)}$ be two solutions of \eqref{boltz} with initial conditions
 $f_0^{(1)}$ and $f_0^{(2)}$ such that $M_1(f_0^{(1)}) = M_1(f_0^{(2)})=m_0$.
Then  for every $t>0$
\begin{equation}\label{ineqdr}
 d_{r}(f_t^{(1)},f_t^{(2)})   \leq d_{r}(f_0^{(1)},f_0^{(2)}) e^{\alpha_r t}
\quad \text{and} \quad
d_{r}^*(f_t^{(1)},f_t^{(2)})   \leq d_{r}^*(f_0^{(1)},f_0^{(2)}) e^{\alpha_r t}.
\end{equation}
\end{theorem}

\begin{proof}  Assume that $d_{r}(f_0^{(1)},f_0^{(2)})<+\infty$, otherwise there is nothing to be proved.
Recall that $\hat p_X(z)= \la z^{X}  \ra=\sum_{m \geq 0} z^m p_X(m)$ and $\hat p_Y(z)= \la z^{Y}  \ra=\sum_{m \geq 0} z^m p_Y(m)$. It is easy to check (cf. Thm. X1.1.1 in \cite{Feller}) that
\begin{equation}\label{barP}
1-\hat p_X(z)= (1-z) \sum_{m \geq 0}  z^m \bar P_X(m),
\end{equation}
where $\bar P_X(m)=P(X > m) =1-\sum_{j=0}^m p_j$. The same result
holds for $\hat p_Y$.
This yields that, for every $z \in (0,1)$, one can write
\[
\frac{ 1-\hat p_X(z)}{1-z}= \sum_m  z^m \bar P_X(m) \leq \sum_m\bar P_X(m) =\la X \ra,
\]
and  $|1-\hat p_Y(z)|/(1-z) \leq \la Y \ra$.
Hence, for every $r >1$,
\[
  \sup_{z \in (0,1)} \frac{|1-\hat p_X(z)|^{r}}{|1-z|^{r}} + \sup_{z \in (0,1)} \frac{|1-\hat p_Y(z)|^{r}}{|1-z|^{r}}
=\la X \ra^{r}+\la Y \ra^{r}=\Delta_r.
\]

Now recall that, by Proposition \ref{propWild},
\[
 \hat f_t^{(i)}(z)=e^{-t} \sum_{n \geq 0} (1-e^{-t})^{n}  \hat  q_n^{(i)}(z),
\]
where
\(
 \hat  q_0^{(i)}(z)=\hat f_0^{(i)}(z)
\)
and, for any $n \geq 1$,
\[
 \hat  q_n^{(i)}(z):=\frac{1}{n} \sum_{j=0}^{n-1} \hat q_j^{(i)}(\hat p_X(z))  \hat  q_{n-1-j}^{(i)}(\hat p_Y(z)).
\]
Using the bound  $|\hat q_j^{(i)}(z)|\leq 1$ for every $z \in (0,1)$ and
$i=1,2$, we obtain
\begin{equation}\label{ineqfourier}
\begin{split}
 d_{r} (q_n^{(1)},q_n^{(2)})  & \leq
  \frac{1}{n} \sum_{j=0}^{n-1}\sup_{z \in (0,1)}
  \frac{|
  \hat q_j^{(1)}(\hat p_X(z))  \hat  q_{n-1-j}^{(1)}(\hat p_Y(z))
- \hat q_j^{(2)}(\hat p_X(z))  \hat  q_{n-1-j}^{(2)}(\hat p_Y(z))|}{|1-z|^{r}}
  \\
& \leq  \frac{1}{n} \sum_{j=0}^{n-1}\sup_{z \in (0,1)}  \frac{|q_{j}^{(1)}(\hat p_X(z)) - q_{j}^{(2)}(\hat p_X(z))}{|1-\hat p_X(z)|^{r}}  \frac{|1-\hat p_X(z)|^{r}}{|1-z|^{r}}\\
&+ \frac{1}{n} \sum_{j=0}^{n-1}
 \sup_{z \in (0,1)}  \frac{|q_{n-1-j}^{(1)}(\hat p_Y(z)) - q_{n-1-j}^{(2)}(\hat p_Y(z))|}{|1-\hat p_Y(z)|^{r}}  \frac{|1-\hat p_Y(z)|^{r}}{|1-z|^{r}}
 \\
&  \leq \frac{\Delta_r}{n} \sum_{j=0}^{n-1} d_{r} (q_j^{(1)},q_j^{(2)}). \\
\end{split}
\end{equation}
Hence, if $d_r(f_0^{(1)},f_0^{(2)})=d_r(q_0^{(1)},q_0^{(2)})<+\infty$ then
$d_{r} (q_n^{(1)},q_n^{(2)}) <+\infty$ for every $n$ and by Lemma \ref{LemmaGamma}
\[
d_{r} (q_n^{(1)},q_n^{(2)}) \leq d_r(f_0^{(1)},f_0^{(1)}) \frac{\Gamma(\Delta_r+n)}{\Gamma(\Delta_r)\Gamma(n+1)}.
\]
Hence, using also \eqref{serieHypergeo},
\[
\begin{split}
d_r(f_t^{(1)},f_t^{(1)}) & \leq \sum_{n \geq 0} e^{-t}(1-e^{-t})^n d_r(q_n^{(1)},q_n^{(2)}) \\
&  \leq d_r(f_0^{(1)},f_0^{(1)}) \sum_{n \geq 0} e^{-t}(1-e^{-t})^n \frac{\Gamma(\Delta_r+n)}{\Gamma(\Delta_r)\Gamma(n+1)}=d_r(f_0^{(1)},f_0^{(1)})  e^{\Delta_r-1}.\\
\end{split}
\]
To recover the second inequality in \eqref{ineqdr}, recall that by \eqref{ineqKX},  for every $r >1$,
\[
 \Delta_r^*:= \sup_{\xi>0} \frac{|k_X(\xi)|^{r}}{|\xi|^{r}}
 + \sup_{\xi > 0} \frac{|k_Y(\xi)|^{r}}{|\xi|^{r}}
=\la X \ra^{r}+\la Y \ra^{r}.
\]
Using this identity, we can repeat the previous part of  proof to conclude.
\end{proof}

\begin{theorem}[\bf{Convergence to steady states}]\label{thmconvergence2}
 Let $\E[X+Y]=1$, and $\la X^r +Y^r \ra<+\infty$ for some $r \in (1,2]$. Then
for every $m_0>0$,  there exists a  unique density $f_\infty$, with mean $m_0$ and finite moment of order $r$ satisfying $f_\infty=Q^+(f_\infty,f_\infty)$.
Moreover, provided  $M_r(f_0)<+\infty$, any solution $f_t$ departing from the initial density $f_0$ with mean $m_0$,  converges in $d_r$-metric
to the corresponding stationary solution $f_\infty $, and
 \[
  d_{r}(f_t,f_\infty)   \leq d_{r}(f_0,f_\infty) e^{-|\alpha_r| t}.
 \]

\end{theorem}

\begin{proof}
Let us observe that $\Delta_r<1$ for every $r>1$ whenever $\la X +Y\ra =1$.
Let $f_0=\delta_{m_0}$ and set $f_{n+1}=Q^+(f_n,f_n)$ for every $n \geq 0$. Since $\la X +Y\ra =1$, one gets $M_1(f_n)=m_0$ for every $n$.
Moreover, by Lemma \ref{inequalitysum} in the Appendix and the fact that $M_1(f_{n})=m_0$ we have
\[
M_r(f_{n+1}) \leq A +\Delta_r M_r(f_{n})
\]
where $A=m_0(\la X^r\ra +\la Y^r \ra)  +m_0^r (2\la X\ra \la Y \ra )^{r/2}$. Iterating on $n$ yields
\begin{equation}\label{ineqMr}
M_r(f_{n+1}) \leq A \sum_{i=0}^n \Delta_r^i +m_0^r \Delta^{n+1}
\leq   \frac{A}{1-\Delta_r} +m_0^r.
\end{equation}
Arguing as in \eqref{ineqfourier} one gets
\[
d_r(f_{n+k},f_{n}) \leq d_r(f_{n+k-1},f_{n-1}) \Delta_r,
\]
and hence
\[
 d_r(f_{n+k},f_{n}) \leq d_r(f_k,f_0) \Delta_r^n.
\]
By Lemma \ref{lemma3} and \eqref{ineqMr}
\[
\sup_k d_r(f_k,f_0) \leq \sup_k \max\{ 2^{r+1}, B^rc_r(M_r(f_k)+M_r(f_0)) \} \leq C.
\]
Hence, the sequence $\{f_n\}_{n\ge0}$ satisfies the Cauchy condition $d_r(f_{n},f_{m}) \leq  \eps$ for every $n,m \geq N(\eps)$.
Since $M_1(f_n)$ is bounded, by Prohorov theorem one obtains that $f_n$ is weakly compact (tight). Combining tightness with the Cauchy condition
one concludes that there is a probability distribution $f_\infty$  such that $\hat f_n \to \hat f_\infty$ (pointwise).
Since
\[
 \hat f_{n+1}(z)= \hat f_n(\hat p_X(z)) \hat f_n(\hat p_Y(z)),
\]
taking $n \to +\infty$, one obtains $\hat f_\infty(z)=\hat f_\infty(\hat p_X(z)) \hat f_\infty(\hat p_Y(z)) $. The property $M_r(f_\infty)<+\infty$ follows from the fact that $\hat f_n \to \hat f_\infty$ and $\sup_n M_r(f_n)<+\infty$
(cf. Thm. 25.11 in \cite{Billingsley}).
Uniqueness of $f_\infty$ (in the class of distribution with finite $r$-moment) follows by Lemma \ref{lemma3} and
Theorem \ref{thmconvergence}.
\end{proof}

\subsection{Evolution of moments} We conclude this Section by giving explicit expressions for the evolution of
the mean and the variance of the solution $f_t$.

\begin{proposition}[\bf{Evolution of the mean}]\label{propmomenti}
Let $f_t$ be the unique solution of equation \eqref{boltz-genfun}
with initial condition $f_0$.
If  $m_0:=M_1(f_0)<+\infty$, then at any subsequent time $t>0$
\begin{equation}\label{consmean-general}
M_1(f_t)=m_0 e^{\alpha_1 t}.
\end{equation}

\end{proposition}

\begin{proof}
First of all note that
\begin{equation}\label{consmean0}
 \sum_v v Q^+(f_0,f_0)(v)=\la X +Y\ra  m_0=\Delta_1 m_0.
\end{equation}
Hence,
using \eqref{wild2}
one gets
\[
  M_1(q_n)=
\frac{\Delta_1}{n} \sum_{j=0}^{n-1}   M_1(q_j) .
\]
By Lemma \ref{LemmaGamma}
\begin{equation}\label{M1qn}
 M_1(q_n)=m_0 \frac{\Gamma(\Delta_1+n)}{\Gamma(\Delta_1)\Gamma(n+1)}.
\end{equation}
Thus, by  \eqref{wild1} and \eqref{serieHypergeo},
\[
M_1(f_t)= e^{-t} \sum_{n \geq 0} (1-e^{-t})^{n}M_1(q_n)  =m_0 e^{-t} \sum_{n \geq 0} (1-e^{-t})^{n}   \frac{\Gamma(\Delta_1+n)}{\Gamma(\Delta_1)\Gamma(n+1)}= e^{t(\Delta_1-1)}.
\]
\end{proof}

\begin{proposition}\label{propmomentir} Let
$\la X^r +Y^r\ra<+\infty$ for some $r \in (1,2]$.
If $f_t$ is the unique solution of equation \eqref{boltz-genfun}
with initial condition $f_0$ such that
$M_1(f_0)=m_0$ and $M_r(f_0)<+\infty$, then
$t \mapsto M_r(f_t)$ is  a continuous function and
\[
M_r(f_t) \leq M_r(f_0)e^{\a_r t} +e^{\a_r t}
H(t).
\]
Here $H(t)$ is given by
\[
H(t)=\int_0^t
  [\Delta_r  m_0 e^{(\alpha_1-\alpha_r) s}
 +
 (2\la X\ra \la Y \ra )^{r/2}  m_0^r e^{(r\alpha_1- \alpha_r) s}   ]ds.
\]
\end{proposition}

\begin{proof}
Let us start by proving that
\begin{equation}\label{powerseriemoment}
t \mapsto M_r(f_t)=e^{-t}\sum_{n \geq 0}(1-e^{-t})^n M_r(q_n)
\end{equation}
is a continuous function on the set
$\{t>0:  M_r(f_t)<+\infty\}$, which turns out to be an open interval.
Identity in equation \eqref{powerseriemoment} is immediate by \eqref{wild1} and Fubini theorem.
Clearly, the series may diverge for some $t<+\infty$.
We claim that
there is $T=T(f_0,X,Y)>0$
such that
$M_r(f_t) \leq +\infty$ for every $t \in [0,T]$.
Using \eqref{wild2} and Lemma \ref{inequalitysum} one can check that
\[
M_r(q_n) =
\frac{1}{n} \sum_{j=0}^{n-1}   M_r(Q^+(q_j,q_{n-1-j}))
= \frac{\Delta_r}{n} \sum_{j=0}^{n-1}  M_r(q_j)+ \frac{B}{n} \sum_{j=0}^{n-1} [M_1(q_j)+M_1(q_j)^r],
\]
for a suitable $B=B(X,Y,r)<+\infty$. Now \eqref{M1qn} and well known asymptotics for the Gamma function give
\[
M_r(q_n) \leq
\frac{1}{n} \sum_{j=0}^{n-1}    M_r(q_j)+B_1 n^\beta,
\]
for suitable $B_1$ and $\beta$. Since Lemma \ref{LemmaGamma2} yields that
$M_r(q_n) \leq K^{n+1}$,   \eqref{wild1} gives
\[
M_r(f_t) = e^{-t} \sum_{n \geq 0} (1-e^{-t})^{n}M_r(q_n) \leq e^{-t} \sum_{n \geq 0} (1-e^{-t})^{n}K^{n+1}.
\]
This shows that $I=\{t>0:  M_r(f_t)<+\infty\}$ is non-empty.
Let us show that $I$ is open. Note that
$t \mapsto e^tM_r(f_t)$ is non-decreasing. Set $s_0:=\sup\{t:M_r(f_t)<+\infty \}$. If $s_0<+\infty$ and
$s_0 \in I$, one can consider the function  $t \mapsto f_{s_0+t}$ as a solution of
\eqref{boltz} with initial condition $f_{s_0}$ such that $M_r(f_{s_0})<+\infty$.
By the previous argument, it must be that $M_r(f_{t+\eps})<+\infty$ for any $\eps \in [0,T(f_{s_0},X,Y)]$, that
gives a contradiction.
 The continuity on $I$ follows by observing that, for $z=(1-e^{-t})$ with $t \in I$,
$\sum_{n \geq 0}z^n M_r(q_n)$ is a convergent power series.

Finally let us prove that, if $s_0:=\sup\{t:M_r(f_t)<+\infty \}<+\infty$
then $\lim_{t \to s_0^{-}} M_r(f_t)=+\infty $.
If $s_0<+\infty$, the previous argument show that $s_0 \not \in I$, that is
$(1-z_0)\sum_{n \geq 0}z_0^n M_r(q_n)=+\infty$ for $z_0=(1-e^{-s_0})$.
Then, Abel theorem yields that $\lim_{z\to z_0^-}(1-z)\sum_{n \geq 0}z^n M_r(q_n)=(1-z_0)\sum_{n \geq 0}z_0^n M_r(q_n)=+\infty $.

Now, by multiplying  equation \eqref{boltz} by $e^t$ and integrating on $[0,t)$ we obtain
\[
e^t f_t(v)=f_0(v)+\int_0^t e^s Q^+(f_s,f_s)(v) ds.
\]
Hence, Fubini Theorem gives
\[
e^t  M_r(f_t)=M_r(f_0)+\int_0^t e^s M_r(Q^+(f_s,f_s)) ds .
\]
Now note that, by Lemma \ref{inequalitysum} in the Appendix
\[
 M_r(Q^+(f_s,f_s)) \leq  \Delta_r(M_1(f_s)+M_r(f_s)) +
 B_r M_1(f_s)^r,
\]
where $B_r=(2\la X\ra \la Y \ra )^{r/2}$ and $\Delta_r=\la X\ra^r + \la Y \ra^r$.
Since $M_1(f_s)=m_0 e^{\alpha_1 s}$, it follows
\[
 M_r(Q^+(f_s,f_s)) \leq  \Delta_r m_0 e^{\alpha_1 s}
 +
 B_r  m_0^r e^{r \alpha_1 s}
    + \Delta_r M_r(f_s).
\]
Denote
$m_r(t)=e^t M_r(f_t)$. We get
\[
m_r(t)\leq  \Delta_r \int_0^t m_r(s) ds +m_r(0)+
\int_0^t e^s [\Delta_rm_0  e^{\alpha_1 s}
 +
 B_r  m_0^r e^{r \alpha_1 s}   ] ds .
\]
Finally, thanks to Gronwall inequality 
which can be applied by virtue of the first part of the proof (see Lemma \ref{gronwall2}), it holds
\[
m_r(t)\leq m_r(0)e^{\Delta_r t} +  \int_0^t e^{\Delta_r (t-s)}
 e^s [\Delta_r m_0  e^{\alpha_1 s}
 +
 B_r  m_0^r e^{r \alpha_1 s}   ] ds,
\]
and $s_0=+\infty$.
\end{proof}

\begin{proposition}[\bf{Evolution of the variance}]\label{propVar}
Let $Var(f_0)<+\infty$
and assume that
\begin{equation}\label{condiziomomentosecondo}
\la X^2+Y^2\ra <+\infty.
\end{equation}
Then
\begin{equation}\label{vart-general}
\begin{split}
  Var(f_t)& =
  M_2(f_0)e^{\a_2t}-m_0^2 e^{2\alpha_1 t}  +
  \beta m_0 \Big [ \frac{e^{\a_1t}-e^{\a_2t}}{\a_1-\a_2}\J\{\a_1 \not  =\a_2 \}
+te^{\a_2t} \J\{\a_1 =\a_2 \}
\Big ]   \\
 &+ 2 \gamma m_0^2 \Big [ \frac{e^{2\a_1t}-e^{\a_2t}}{2\a_1-\a_2}\J\{2\a_1 \not =\a_2 \}
+te^{\a_2t} \J\{2\a_1 =\a_2 \}
\Big ]
   \\
\end{split}
\end{equation}
where
\[
\beta := Var(X)+Var(Y) \quad \text{and} \quad \gamma := \la X \ra \la Y \ra.
\]
\end{proposition}

\begin{proof} Recall that, by Proposition \ref{propmomentir},  $t \mapsto M_2(f_t)$
is a continuous function.
Using  formula   \eqref{squarethesum} in Lemma \ref{inequalitysum} we obtain
\begin{equation}\label{sumsquare}
  M_2(Q^+(f_s,f_s)) =M_2(f_s)[\la X \ra^2+\la Y \ra^2]+ M_1(f_s) \beta+2 \gamma M_1(f_s)^2.
\end{equation}
Combining \eqref{boltz-integral}, \eqref{consmean-general}   and \eqref{sumsquare}
one gets
\begin{equation}\label{ev_of_variance}
M_2(f_t)-M_2(f_0)=\int_0^t (\alpha_2 M_2(f_s)+c(s))ds,
\end{equation}
with
\[
c(s)=m_0 \Big ( \beta e^{\a_1s} +2 \gamma m_0 e^{2\a_1 s}  \Big) .
\]
Solving the equation one obtains
\[
M_2(f_t)=e^{\a_2 t} \Big (M_2(f_0)+ \int_0^t c(s) e^{-\alpha_2 s}ds \Big ),
\]
which implies
\[
Var(f_t)=e^{\a_2 t} \Big (M_2(f_0)+ \int_0^t c(s) e^{-\alpha_2 s}ds \Big )-m_0^2 e^{2\a_1t}.
\]
Simple computations then give \eqref{vart-general}.
\end{proof}

Note that from $\la X +Y\ra =1$ and
 $\la X^2+Y^2\ra <+\infty$, it follows $\a_2<0=\a_1$
and $2\gamma=-\alpha_2$. Then \eqref{vart-general} implies
\begin{equation}\label{vart}
  Var(f_t)=\frac{(Var(X)+Var(Y)) m_0}{|\alpha_2|}+[Var(f_0)-\frac{(Var(X)+Var(Y)) m_0}{|\alpha_2|}] e^{-|\a_2| t},
\end{equation}
for every $t>0$.

\section{The case of increasing mean}\label{sec:scaled}

The asymptotic behavior of the solution to the Boltzmann equation \fer{boltz} when  the mean value is preserved 
in time has been fully described in Theorem \ref{thmconvergence2}. In this case, in fact, existence 
of a steady solution together with its main properties follows.

A completely different situation arises when  $\alpha_1=\la X+Y\ra-1>0$, as in  Case
$2$. Now the mean value diverges to $+\infty$, and the usual way to extract
information about the large-time behavior of the solution \cite{Pareschi06} is to
scale the solution of equation \eqref{boltz-genfun}, or equivalently
\eqref{boltz-laplace}, with respect to the mean, in such a way to maintain the mean
constant. This allows to look for non-trivial asymptotics profiles of
\[
\tilde h(\xi,t)=\tilde f_t\Big(\frac{\xi}{m(t)}\Big),
\]
where
\[
m(t):=M_1(f_t)=m_0 e^{\a_1 t}.
 \]
It is easy to verify that the Laplace transform $\tilde h$ of the scaled density
satisfies the new equation
\begin{equation}\label{boltz-laplace-scaled}
 \frac{\partial }{\partial t} \tilde h_t(\xi) +  \tilde h_t(\xi) +
\alpha_1  \xi  \frac{\partial }{\partial \xi}  \tilde h_t(\xi) =  \tilde h_t\Big(-m(t) k_X\Big (\frac{\xi}{m(t)}\Big)\Big) \tilde h_t \Big(-m(t)  k_Y\Big (\frac{\xi}{m(t)} \Big )\Big).
\end{equation}

Note that if $f^{(1)}$ and $f^{(2)}$ are two solutions of  \eqref{boltz}, with initial conditions $f_0^{(1)}$ and $f_0^{(2)}$, while $h^{(1)}$ and $h^{(2)}$ are
the respective scaled solutions, then
\[
 d_{r}^*(h_t^{(1)},h_t^{(2)})  \leq d_{r}^*(f_t^{(1)},f_t^{(2)}) m(t)^{-r}
 =d_{r}^*(f_t^{(1)},f_t^{(2)})m_0^{-r} e^{-tr\alpha_1},
\]
and  \eqref{ineqdr}
yields that
\[
 d_{r}^*(h_t^{(1)},h_t^{(2)})
  \leq m_0^{-r} d_{r}^*(h_0^{(1)},h_0^{(2)}) e^{t(\alpha_r-r\alpha_1 )}.
\]
In view of the previous bound, it appears natural  to assume that
\begin{equation}\label{condizionealphar}
(\alpha_r/r-\alpha_1)<0 \quad \text{for some $r>1$}.
\end{equation}
Condition \eqref{condizionealphar} is reminiscent of the analogous condition for the
existence of self-similar solutions of Kac like kinetic models, see
\cite{BaLa,BoCeGa,Pareschi06}. Indeed, as we shall see, there is a precise connection
between the asymptotic profile of the scaled solution of
 \eqref{boltz} and the asymptotic profile  of the solution of a particular Kac like equation.
It is important to remark that, since $\alpha_r=\E[X]^r+\E[Y]^r-1$, condition
\eqref{condizionealphar} is equivalent to ask that $r \mapsto \alpha_r/r$ is
decreasing in $r=1$, i.e. $ (\alpha_r/r)'<0$ for $r=1$. Computing  $(\alpha_r/r)'$ one
obtains that \eqref{condizionealphar} is equivalent to
\begin{equation}\label{condizionealphar2}
 \E[X](\log(\E[X])-1)+\E[Y](\log(\E[Y])-1)-1<0.
\end{equation}
Clearly the previous relation is true if $\E[Y]<1$ and $\E[X]<1$, while is clearly false if one of the two expectations is bigger than $e$. In Figure \ref{Fig1}
a numerical evaluation of the region where  \eqref{condizionealphar} is verified and $\E[X]+\E[Y]>1$ is reported.
Additional information on the function $r \mapsto \alpha_r/r-\alpha_1$ can be found in
Lemma 3.10 in \cite{Pareschi06}.

\begin{figure}[ht]
\begin{centering}
\begin{tabular}{cc}
\includegraphics[scale=0.4]{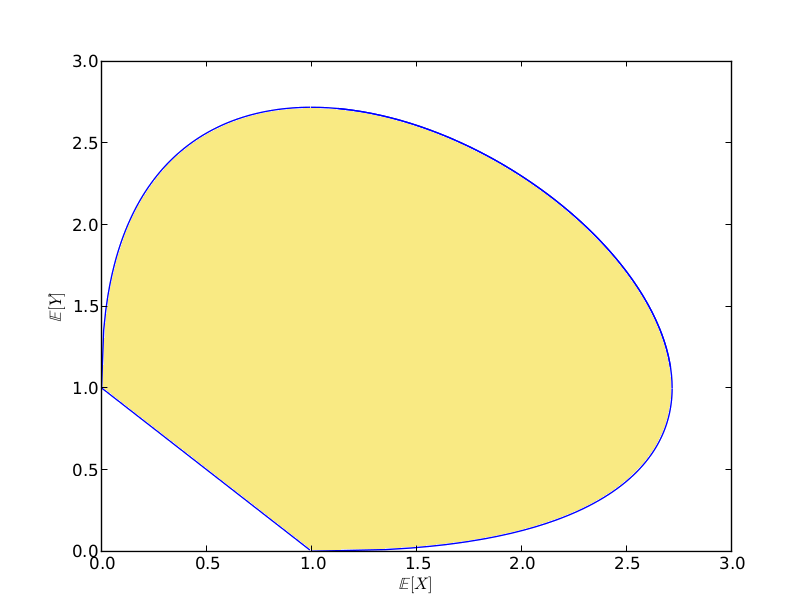}
\end{tabular}
\end{centering}
\caption{The interior  of the curve  (colored) represents the region in the plane $(\E[X],\E[Y])$ where \eqref{condizionealphar} is verified  and $\E[X]+\E[Y]>1$. } \label{Fig1}
\end{figure}

Recalling that  here $m(t) \to +\infty$, using \eqref{D6}, one can write
\[
-m(t)  k_X(\xi/m(t) )=\xi \la X \ra + O(\xi^r m(t)^{1-r}),
\]
and hence,
\[
\lim_{t \to +\infty} \Big ( - m(t)  k_X(\xi/m(t) \Big) =\xi \la X \ra .
\]
The same result holds for $Y$. In this way one can see that  (at least formally)  $\tilde h_t$ converges to the solution  $\tilde h_\infty(\xi)=\int_{[0,+\infty)} e^{-\xi v}
h_\infty(dv)$ of
\begin{equation}\label{boltz-laplace-res-inf}
 \tilde h_\infty(\xi)  +\alpha_1  \xi  \frac{\partial }{\partial \xi}  \tilde h_\infty(\xi) =
 \tilde h_\infty(\xi \la X \ra  ) \tilde h_\infty(\xi \la Y \ra  ).
\end{equation}
For a precise statement see next Proposition \ref{Prop_conscaled}. Now following
\cite{BaLa,BaTo2,BoCeGa}, \eqref{boltz-laplace-res-inf} is equivalent to the integral
equation
\begin{equation}\label{boltz-laplace-res-inf2}
 \tilde h_\infty(\xi)  =\int_0^1
 \tilde h_\infty(\xi \la X \ra \tau^{\alpha_1}  ) \tilde h_\infty(\xi \la Y \ra\tau^{\alpha_1}  ) d\tau .
\end{equation}
This shows that $h_\infty$ is the fixed point of the particular {\it smoothing transformation}:
\[
T(h_\infty,h_\infty)=\text{\bf Law} \Big (\la X \ra U^{\alpha_1} V'+\la Y \ra U^{\alpha_1} V'' \Big )
\]
where $U,V',V'$ are independent, $U$ is uniformly distributed on $(0,1)$, and $V',V''$ are distributed with law $h_\infty$.
Note that
\[
\la \la X \ra U^{\alpha_1}+\la Y \ra U^{\alpha_1} \ra =\frac{\la X \ra +\la Y \ra}{1+\alpha_1}=1.
\]
Positive fixed points of smoothing transformations are deeply studied, mainly in
connection with the so-called Branching Random Walk, see e.g.
\cite{AlsmeyerMeiners13,liu98}. It is worth noticing that in general $-$  with the
exception of few cases, see  \cite{BaTo,BaTo2}  $-$ there is no analytical
expression of the law the of the fixed point of a given smoothing transformation. Even
in our case it does not seem easy to obtain an explicit expression of $h_\infty$. An
exception is the very special case when $E[X]=E[Y]=1$, where it is immediate to see
that the solution is given by the exponential distribution
\[
\hat h_\infty(\xi)=\frac{1}{1-\xi \th}.
\]
Nevertheless, it is always possible to recursively determine the exact expression of
the integer moments of $h_\infty$. Indeed, using the fact that
$h_\infty=T(h_\infty,h_\infty)$ and the Newton binomial formula, if $m_i:=\int v^i
h_\infty(dv)<+\infty$ and $m_1=\int v h_\infty(dv)=1$, one gets
 \begin{equation}\label{recursionmoment}
 m_i=\frac{1}{(\E[X]+\E[Y]-1)i+1-\E[X]^i-\E[Y]^i}   \sum_{j=1}^{i-1}\binom{i}{j} \E[X]^i\E[Y]^{i-j}  m_j m_{i-j}.
\end{equation}
As we shall see in Proposition \ref{probBaLa} below, $\int v^r h_\infty(dv)<+\infty$
if and only if $\a_r/r<\a_1$. Hence, $m_i<+\infty$ if and only if $(\E[X]+\E[Y]-1)i+1-\E[X]^i-\E[Y]^i>0$
so that \eqref{recursionmoment} is well-defined.
Note that when at least one of the two moments $\E[X]$ and $\E[Y]$ is strictly bigger than 1, then
there is $\bar r<+\infty$ such that $(\E[X]+\E[Y]-1)\bar r+1-\E[X]^{\bar r}-\E[Y]^{\bar r}<0$ and hence
$\int v^{\bar r} h_\infty(dv)=+\infty$. This means that in this case
$h_\infty$ has heavy tails. On the contrary if both $\E[X]$ and $\E[Y]$ are smaller than 1, then
$h_\infty$ has moments of all orders.

The fact that the steady state of \eqref{boltz-laplace-scaled} satisfies \eqref{boltz-laplace-res-inf2} suggests that,
for a better understanding of the asymptotic behavior  of equation \eqref{boltz-laplace-scaled},
one can first to consider the large-time behavior of the equation
\begin{equation}\label{boltzLRscaled}
 \frac{\partial }{\partial t} \tilde h_t(\xi) +  \tilde h_t(\xi)+ \alpha_1  \xi  \frac{\partial }{\partial \xi}  \tilde h_t(\xi) =  \tilde h_t\Big(\la X\ra \xi \Big)
 \tilde h_t \Big( \la Y \ra \xi \Big),
\end{equation}
which is satisfied by the scaled solution $\tilde h(\xi,t)=\tilde f_t({e^{-\a_1t}}\xi/m_0)$
of the equation
\begin{equation}\label{boltzLR}
 \frac{\partial }{\partial t} \tilde f_t(\xi) +  \tilde f_t(\xi)
 =  \tilde f_t\Big(\la X\ra \xi \Big)
 \tilde f_t \Big( \la Y \ra \xi \Big).
\end{equation}
 Equation \fer{boltzLR} is associated to the very simple linear collision rule
 \begin{equation}\label{coll1}
 V_i'= {V_i} \la X \ra + {V_j} \la Y \ra, \qquad
 V_j'=  V_j  \la X \ra   + V_i  \la Y \ra
\end{equation}
and it is a special case of a general class of kinetic equations which have been
deeply investigated, see for instance in \cite{BaLa,BaLaMa,BoCeGa,MaTo,PeRe}. It is
immediate to realize that the interactions \fer{coll1} correspond to a mean version of
the interactions \fer{coll0}, and can be formally obtained from \fer{coll0} by
substituting $X$ and $Y$ with their mean values. Of course, this is consistent only if
the mean values of $X$ and $Y$ are integer numbers. For this equation, one can resort
to the results of \cite{BaLa} to obtain

\begin{proposition}[\bf{\cite{BaLa}}]\label{probBaLa} Let $\la X+Y\ra>1$, and let
  \eqref{condizionealphar} hold true for some $r >1$. Then, there exists a unique solution to \eqref{boltz-laplace-res-inf2}  such that $\int_{[0,+\infty)} v h_\infty(dv)=1$.
Moreover, this solution satisfies $\int_{[0,+\infty)} v^q h_\infty(dv)<+\infty$ for $q>1$
if and only if $\alpha_q/q<\alpha_1$.
Finally, if $\int v f_0(dv)=m_0<+\infty$ and $\tilde f_t$ is a solution of  \eqref{boltzLR},
then $\int v f_t(dv)=m_0e^{\alpha_1 t}$ and $\tilde f_t({e^{-\a_1t}}\xi/m_0)$ converges  to $\tilde h_\infty(\xi)$ for every $\xi>0$ as $t \to \infty$.
\end{proposition}

\begin{proof}
The result follows from Proposition 2.1 and Theorem 2.2  in \cite{BaLa}. We have only to note that the spectral function $\mu(\cdot)$ defined in Section 2.2
of \cite{BaLa} in this case coincides with $r \mapsto \alpha_r/r -\alpha_1$.
\end{proof}

One can adapt some results in \cite{BaLa,MaTo} to obtain

\begin{proposition}\label{propMaTo} Let $\la X+Y\ra>1$ and assume that \eqref{condizionealphar} holds true
for some $r \in (1,2)$. Let $\tilde f_t$ be a solution of  \eqref{boltzLR}  with
initial condition $f_0$ with finite $r$ moment. Then
\begin{equation}\label{boundmomentMaTo}
M_r(f_t) \leq C_1 e^{\alpha_r t}.  
\end{equation}
Moreover, if $\tilde h_t(\xi)=\tilde f_t({e^{-\a_1t}}\xi/m_0)$, $m_0=\int v f_0(dv)$, and $\tilde h_\infty(\xi)$
is the same as in {\rm Proposition \ref{probBaLa}}, then
\[
d_r^*(h_t,h_\infty) \leq C_2 e^{(\a_r -r\a_1)t},
\]
for a suitable constant $C_2=C_2(f_0,h_\infty)$.
\end{proposition}

\begin{proof}
The first part of the thesis is a minor modification of Theorem 3.2 in \cite{MaTo},
observing that in this case the function $\mathfrak{S}(r)$ --defined in equation (32) of \cite{MaTo} -- coincides with $\la X\ra^r + \la Y \ra^r-1=\alpha_r$.
The last part of the thesis follows by a simple modification of
Thm. 2.2 in \cite{MaTo}. See also Proposition 3.12 in \cite{BaPe}
\end{proof}

In the next proposition we show that, under suitable assumptions,  the scaled solutions of
\eqref{boltz} and \eqref{boltzLR} merge when $t \to +\infty$. Combining this result with
the known asymptotics for the scaled solution of \eqref{boltzLR}
one obtains the asymptotic behavior of the scaled solution of  \eqref{boltz} in Case $2$.

\begin{proposition}[\bf{Scaled solution}]\label{Prop_conscaled} Let $\la X+Y\ra>1$ and assume that \eqref{condizionealphar} holds true
for some $r \in (1,2)$.
Let $\tilde f^{(1)}$ be a solution of \eqref{boltz-laplace} with initial conditions $\tilde f_0$ with finite $r$-th moment and $\tilde f^{(2)}$ be a solution
of  \eqref{boltzLR} with the same initial condition $\tilde f_0$ with mean $m_0$.
Set $\tilde h^{(i)}_t(\xi)=\tilde f_t^{(i)}({e^{-\a_1t}}\xi/m_0)$ for $i=1,2$.
Then,
\[
 d_{r}^*(h_t^{(1)},h_t^{(2)})
  \leq C(f_0)[t e^{(\a_r -r\a_1)t} +  e^{-\a_1 rt}].
\]
In particular, $h^{(1)}_t$ converges weakly to the probability distribution
$h_\infty$ given in {\rm Proposition \ref{probBaLa}}.
\end{proposition}


\begin{proof}
We have
\[
\begin{split}
\xi^{-r}\Big | & \tilde f_t^{(1)}\Big (-k_X(\xi)\Big) \tilde f_t^{(1)}\Big(-k_Y(\xi)\Big)
- \tilde f_t^{(2)}\Big(-\la X\ra \xi\Big) \tilde f_t^{(2)}\Big(-\la Y\ra \xi\Big)\Big|  \\
& \leq \xi^{-r}\Big | \tilde f_t^{(1)}\Big (-k_X(\xi)\Big)\tilde f_t^{(1)}\Big(-k_Y(\xi)\Big)
 -\tilde f_t^{(2)}\Big (-k_X(\xi)\Big)\tilde f_t^{(2)}\Big (-k_Y(\xi) \Big) \Big |
\\
&
+ \xi^{-r}\Big |
\tilde f_t^{(2)}\Big (-k_X(\xi)\Big)\tilde f_t^{(2)}\Big (-k_Y(\xi) \Big)
-\tilde f_t^{(2)}\Big (-k_X(\xi)\Big)\tilde f_t^{(2)}\Big(\la Y\ra \xi\Big)
\Big |
\\
& +\xi^{-r} \Big | \tilde f_t^{(2)}\Big (-k_X(\xi)\Big)\tilde f_t^{(2)}\Big(\la Y\ra \xi\Big)
- \tilde f_t^{(2)}\Big(\la X\ra \xi\Big) \tilde f_t^{(2)}\Big(\la Y\ra \xi\Big)\Big|
\\
& =:A_1+A_2+A_3.
\\
\end{split}
\]
Now, arguing as in the proof of Theorem  \ref{thmconvergence},
\[
A_1
 \leq \Delta_r d_r^*(f_t^{(1)},f_t^{(2)}).
\]
Moreover, using \eqref{D7}
\[
A_2 \leq \xi^{-r}\Big |
\tilde f_t^{(2)}\Big (-k_Y(\xi) \Big)
-\tilde f_t^{(2)}\Big(\la Y\ra \xi\Big)
\Big | \leq C(Y) [1+M_r(f_t^{(2)})]
\]
and analogously
\[
A_3 \leq C(X) [1+M_r(f_t^{(2)})].
\]
At this stage recall that, by Proposition \ref{propMaTo},
\[
M_r(f_t^{(2)}) \leq C_1 e^{\alpha_r t}.
\]
Putting all estimates together one gets
\begin{equation}\label{ineqLRvBr}
\begin{split}
H_t(\xi):=\xi^{-r} \Big |  & \tilde f_t^{(1)} \Big (-k_X(\xi)\Big) \tilde f_t^{(1)}\Big(-k_Y(\xi)\Big)
- \tilde f_t^{(2)}\Big(-\la X\ra \xi\Big) \tilde f_t^{(2)}\Big(-\la Y\ra \xi\Big)\Big|
\\ &
\leq
\Delta_r d_r^*(f_t^{(1)},f_t^{(2)}) +C_3[1+ e^{\alpha_r t}].
\\
\end{split}
\end{equation}
Now, consider that
\[
 \frac{\partial }{\partial t} (e^{t} \tilde f_t^{(1)}(\xi) )
 =  e^t\tilde f_t^{(1)}\Big(-k_X(\xi) \Big)
 \tilde f_t^{(1)} \Big( -k_Y(\xi) \Big),
\]
and
\[
 \frac{\partial }{\partial t} (e^{t} \tilde f_t^{(2)}(\xi) )
 = e^t \tilde f_t^{(2)}\Big(\la X\ra \xi \Big)
 \tilde f_t^{(2)} \Big( \la Y \ra \xi \Big).
\]
Last, consider that $\tilde f_0^{(1)}=\tilde f_0^{(2)}$, and, using  \eqref{ineqLRvBr}
\[
e^t \xi^{-r} \Big |\tilde f_t^{(1)} (\xi)-\tilde f_t^{(2)} (\xi)\Big|
\leq
\int_0^t e^s H_s(\xi)ds \leq  \int_0^t e^s [\Delta_r d_r^*(f_s^{(1)},f_s^{(2)}) +C_3(1+ e^{\alpha_r s})]ds.
\]
Hence,
\[
e^t d_r^*(f_t^{(1)},f_t^{(2)}) \leq \Delta_r  \int_0^t  e^s d_r^*(f_s^{(1)},f_s^{(2)}) ds
 + C_3 \int_0^t  [e^s+e^{(1+\alpha_r) s})]ds.
\]
By Lemma \ref{lemma3},
\[
d_r^*(f_t^{(1)},f_t^{(2)})\leq c_r [M_r(f_t^{(1)})+M_r(f_t^{(2)})]
\]
and, by Propositions \ref{propmomentir} and \ref{propMaTo}, $t \mapsto M_r(f_t^{(1)})+M_r(f_t^{(2)})$ is a locally bounded function.
Hence, Gronwall Lemma  gives
\[
d_r^*(f_t^{(1)},f_t^{(2)}) \leq C_3  e^{(\Delta_r-1)t} \int_0^t [
e^{(1-\Delta_r)s}+1)]  ds.
\]
If now $\alpha_r\not=0$, it holds
\[
d_r^*(h_t^{(1)},h_t^{(2)})
\leq C_3 e^{-\a_1 rt} [ t e^{\alpha_r t}+ \frac{1-e^{\alpha_rt}}{-\alpha_r} ]
\leq C_4 [t e^{(-\a_1 r+\a_r)t} +  e^{-\a_1 rt}].
\]
In the remaining case $\alpha_r=0$ it holds $d_r^*(h_t^{(1)},h_t^{(2)})\leq C_3 e^{-\a_1 rt}$.
\end{proof}

\section{The quasi-invariant limit}\label{Sec:grazing}

In this section we give  precise statements and proofs  of the limit procedure
considered in Section \ref{S:Quasiinvariant}.

\begin{proposition}\label{prop:limite}
Let $\hat f_{t,\eps}(z)$ be the solution of \eqref{boltz-genfun} for $X$ and $Y$ defined in  \eqref{variablegrazing}
where $\E[\tilde X^2+\tilde Y^2]<+\infty$ and   $f_0$ has finite second moment
 and mean equal
to $m_0$.
Then,
$
\hat g(t,z)=\lim_{\eps \to 0} \hat f_{t/\eps,\eps}(z)
$
satisfies
\begin{equation}\label{equationgrazing}
\begin{split}
&\partial_t  \hat g(t,z) = b_1 ( \hat p_{\tilde X}(z) - z) \partial_z \hat g(t,z)+b_2
m_0 e^{\bar \alpha_1 t}(\hat p_{\tilde Y}(z)-1) \hat g(t,z),
\\
& \hat g(0,z)=\hat f_0(z). \\
\end{split}
\end{equation}
for $(t,z) \in (0,+\infty) \times  [0,1]$,
where   $\bar \alpha_1=b_1(\E[\tilde X]-1)+b_2\E[\tilde Y]$. Moreover, the initial value problem for equation \fer{equationgrazing} has a unique solution, and
\[
M_1(g(t))=m_0  e^{\bar \alpha_1 t}.
\]

\end{proposition}

\begin{proof}
Existence and uniqueness of a unique bounded solution in $Q=(0,1)\times (0,T)$, for every $T >0$ to equation  \eqref{equationgrazing} is known (cf. \cite{DiPernaLions89} for a detailed study of this type of linear equations).

 Simple computations show that
\begin{equation*}\label{xyeps1}
 \alpha_2:=
\la X \ra^2 +\la Y \ra^2 -1= 2\eps (\laa \tilde X \raa-1)b_1 +\eps^2((\E[\tilde X-1])^2 b_1^2 +  (\E[Y])^2 b_2^2),
\end{equation*}
while
\begin{equation*}\label{xyeps2}
\beta:=Var(X)+Var(Y)= \eps(\laa (\tilde X-1)^2 \raa b_1 +  \laa  \tilde Y^2 \raa b_2 )-\eps^2((\E[\tilde X-1])^2 {b_1}^2 +  (\E[Y])^2 {b_2}^2),
\end{equation*}
and
\begin{equation*}\label{xyeps3}
 \gamma:= \E[X]\E[Y]=\eps b_2 \E[\tilde Y](1+\eps b_1(\E[\tilde X]-1)).
\end{equation*}
Recall that $\alpha_1=\la X \ra +\la Y \ra-1=\eps \bar \alpha_1$.
Hence, owing to \eqref{vart-general}, one finds that
\begin{equation}\label{xyeps4}
 Var(f_{t,\eps})\leq G(\eps t),
\end{equation}
for a suitable function $G$ bounded on any fixed interval of time $[0,T]$.
In addition we have
\[
| \partial_z^2 \hat f_{t,\eps}(z)|= |\sum_{v \geq 2} v(v-1) z^{v-2} f_{t,\eps}(v)|
\leq \sum_{v \geq 2} v^2 f_{t,\eps}(v) =Var(f_{t,\eps})+M_1(f_{t,\eps})^2.
\]
Recalling that that $M_1(f_{t,\eps})=m_0e^{\bar \alpha_1 \eps t}$ and using
\eqref{xyeps4}, for every $T<+\infty$ one gets the bound
\[
 \sup_{t: t\eps \leq T}| \partial_z^2 \hat f_{t,\eps}(z)| \leq  C,
\]
where $C=C(f_0,\tilde X,\tilde Y,T)$ is independent of $\eps\leq \eps_0$ and $z$.
Hence $\partial_z f_{t,\eps}(z)$ is locally uniformly (in $t$ and $\eps$) Lipshitz in $z$, which means that
\begin{equation}\label{lipunif}
 \Big | \partial_z f_{t,\eps}(z)|_{z=z_0}- \partial_z f_{t,\eps}(z)|_{z=z_1} \Big | \leq L(T)
 |z_0-z_1|.
\end{equation}
for every $z_0,z_1$ in $[0,1]$ and every $t>0$ such that $t\eps<T$ and $\eps \leq \eps_0$.

To prove convergence, write
\[
 \hat f_t(\hat p_X(z))=\hat f_t(z) +\eps b_1 ( \hat p_{\tilde X}(z) - z)  \partial_z \hat f_t(z) +\eps b_1  R_{1,\eps}(t,z),
\]
and
\[
 \hat f_t(\hat p_Y(z))=\hat f_t(1)+\eps b_2 (\hat p_{\tilde Y}(z)-1) \partial_z \hat f_t(1) +\eps b_2 R_{2,\eps}(t,z).
\]
In the previous equations,
\[
 R_{1,\eps}(t,z) =( \hat p_{\tilde X}(z) - z) (  \partial_z \hat f_t(z_{1,\eps,t})- \partial_z \hat f_t(z)),
\]
for some $z_{1,\eps,t} \in [0,1]$ such that $|z-z_{1,\eps,t}| \leq \eps b_1  |\hat p_{\tilde X}(z) - z|$ and
\[
 R_{2,\eps}(t,z) =( \hat p_{\tilde Y}(z) - 1) (  \partial_z \hat f_t(z_{2,\eps,t})- \partial_z \hat f_t(1)),
\]
for some $z_{2,\eps,t}$ such that $0 \leq 1+b_2 \eps (\hat p_{\tilde Y}(z) - 1) \leq  z_{2,\eps,t} \leq 1 $.
Consequently
 \[
\hat f_t (\hat p_X(z) )  \hat f_t(\hat p_Y(z)) -\hat f_t(z) =
\eps b_2 \hat f_t(z)(\hat p_{\tilde Y}(z)-1) m(t)  + \eps b_1 ( \hat p_{\tilde X}(z) - z)   \partial_z \hat f_t(z) + \eps R_{3,\eps}(t,z).
 \]
Here, $m(t)= \partial_z \hat f_t(1)$ and
\[
\begin{split}
 R_{3,\eps}(t,z)& = \eps b_1 b_2( \hat p_{\tilde X}(z) - z)    ( \hat p_{\tilde Y}(z) - 1)\partial_z \hat f_t(z)  m(t)
+\eps b_1b_2( \hat p_{\tilde X}(z) - z) \partial_z \hat f_t(z)  R_{2,\eps}(t,z) \\
& + b_2 R_{2,\eps}(t,z) \hat f_t(z) + R_{1,\eps}(t,z) b_1[1+\eps b_2(\hat p_{\tilde Y}(z)-1)  m(t)  +\eps b_2 R_{2,\eps}(t,z) ]. \\
\end{split}
\]
Let us set
 $\hat g_{t,\eps}(z)=\hat f_{t/\eps}(z)$. Then
\begin{equation}\label{eqscaled2}
 \partial_t \hat g_{t,\eps}(z) =\hat g_{t,\eps}(z) b_2(\hat p_{\tilde Y}(z)-1) M_1(g_{t,\eps})  +
b_1(\hat p_{\tilde X}(z) - z)   \partial_z \hat g_{t,\eps}(z) +  R_{3,\eps}(t/\eps,z).
\end{equation}
Using \eqref{lipunif}, for a suitable constant $C_1(T)$ we obtain
\[
 \sup_{(z,t) \in Q} \sum_{i=1,2} |R_{i,\eps}(t/\eps,z)| \leq   \eps C_1(T),
\]
so that
\begin{equation}\label{R3eps}
  \sup_{(z,t) \in Q} |R_{3,\eps}(t/\eps,z)| \leq  \eps C_2(T).
\end{equation}
Recalling now  that for every $z \in [0,1]$ and every $t\in [0,T]$
\[
 \partial_z \hat g_{t,\eps}(z) \leq  M_1(g_{t,\eps}) \leq C_3(T),
\]
using this with \eqref{R3eps} in equation \eqref{eqscaled2} we obtain
\(
| \partial_t   \hat g_{t,\eps}(z)| \leq C_4(T)
\).
Summarizing, for a suitable bounded positive constant $C_5(T)$ we have
\[
\sup_{(z,t) \in [0,1] \times [0,T]}  | \partial_z \hat g_{t,\eps}(z)| + \sup_{(z,t) \in [0,1] \times [0,T]} |\partial_t \hat g_{t,\eps}(z)| \leq C_5(T).
\]
Combining  Ascoli-Arzel\'a theorem with the Banach-Alaoglu-Bourbaki theorem
one concludes that there is a subsequence $\hat g_{\cdot,\eps_n}(\cdot)$ converging strongly in $C^0([0,1]\times [0,T])$ (for every $T$) and weakly-*
in $W^{1,\infty}(Q)$ to a function $\hat g$.

Since $R_{3,\eps} \to 0$ uniformly by \eqref{R3eps}, one can pass to the limit in \eqref{eqscaled2}
to obtain that $\hat g$ is  a solution of \eqref{equationgrazing}.
Next, uniqueness of the solution of \eqref{equationgrazing} gives that the whole family $\hat g_{\cdot,\eps}$
converges to $\hat g$ strongly in $C^0([0,1]\times [0,T])$, for every $T>0$. Finally, since
the mean and variance of $g_{t,\eps}$ are uniformly bounded with respect to $\eps$, the  sequence $\{g_{t,\eps}\}_{\eps \ge 0}$ is tight, 
and the limit function $g(\cdot,t)$ is also a p.g.f. with finite variance.
\end{proof}

Define  $\phi_X(z)=\sum_{k \geq 0} z^k P\{\tilde X >k\}$ and  $\phi_Y(z)=\sum_{k \geq 0} z^k P\{\tilde Y >k\}$.
Setting $m=\laa \tilde X \raa$, if $b_1(\E[\tilde X]-1)+b_2\E[\tilde Y]=0$, one can write
$(1-m)b_1/b_2=\laa \tilde Y \raa$ so that $ P^*(z):=m^{-1} \phi_X(z)$ and $Q^*(z):=b_2/b_1(1-m)^{-1} \phi_Y(z)$
are p.g.f.. It is worth noticing that, if $b_1=b_2=1$, $P^*$ and $Q^*$ are the p.g.f. of the so called {\it renewal distribution} of
$p_X$ and $p_Y$.
Recall also that given a probability density $g$ on $\N$ with finite mean
$M_1(g)=\sum_k k g(k)$, its {\it size-biased version} is
defined by
\[
 g^*(k):=\frac{k g(k)}{M_1(g)}.
\]

\begin{proposition}\label{prop:steady_grazing} Under the same assumptions of {\rm Proposition \ref{prop:limite}}, if $b_1(\E[\tilde X]-1)+b_2\E[\tilde Y]=0$, then
 the unique $C^1$-function
$\hat g_{\infty}(z)$ with $\hat g_{\infty}(1)=1$ which is a stationary solution of \eqref{equationgrazing}, i.e. the unique solution
of \eqref{stat_grazing},
is given by
\begin{equation}\label{ginfty}
 \hat g_\infty(z)=\exp \Big \{ -m_0 \frac{b_2}{b_1} \int_{z}^1 \frac{(1-\hat p_{\tilde Y}(s))}{(\hat p_{\tilde X}(s) - s)}ds  \Big \}.
\end{equation}
Moreover, $\hat g_{\infty}$ is the p.g.f. of an infinite divisible
distribution $g_{\infty}$ over the non-negative integers.
Finally,  $g_{\infty}$ is the density of a random variable $V_\infty$ satisfying
\begin{equation}\label{sizebiasedinfdiv}
 V_\infty^* \stackrel{d}{=}  V_\infty + 1+ \zeta+\sum_{k=1}^{\gamma} \xi_k
\end{equation}
where $V_\infty^*$ is the size biased version of $V_\infty$,
 $\zeta$ has p.g.f. $Q^*$, $\xi_k$ has p.g.f. $P^*$ for every $k \geq 1$,
$\gamma$ has geometric distribution of parameter $m=\laa \tilde X \raa$ and $V_\infty,\zeta,\gamma,\xi_1,\xi_2,\dots$
are stochastically independent.
\end{proposition}

\begin{proof}
 A stationary distribution $\hat g_\infty$  of  \eqref{equationgrazing} must satisfy \eqref{stat_grazing}.
That is
\[
  \partial_z \hat g_\infty(z)=\frac{m_0 b_2(1-\hat p_{\tilde Y}(z))}{b_1(\hat p_{\tilde X}(z) - z)} \hat g_\infty(z),
\]
that gives \eqref{ginfty}.
Now note that, using \eqref{barP} and the analogous relation for $Y$, one gets
\[
 R(z):=\frac{b_2(1-\hat p_{\tilde Y}(z))}{b_1(\hat p_{\tilde X}(z) - z)} =\frac{b_2 \phi_Y(z)}{b_1(1-\phi_X(z))}
=\frac{(1-m)Q^*(z)}{1-m P^*(z)}
= G_m( P^*(z)) Q^*(z)
\]
where
\[
 G_m(s)=\frac{1-m}{1-ms}
\]
is the p.g.f. of a geometric random variable of parameter $m$. In other words $R(z)$ is the p.g.f. of
the random variable
\[
  \zeta+\sum_{k=1}^{\gamma} \xi_k.
\]
We proved that
\[
 \hat g_{\infty}(z)=\exp \Big \{ -m_0 \int_{z}^1 R(s) ds  \Big \},
\]
where $R$ is a p.g.f. and hence $m_0R$ is an  absolutely monotone function. This show
that $\hat g_{\infty}(z)$ is a p.g.f. of an infinite divisible distribution
$g_{\infty}$ over the non-negative integers. Indeed a function  $h$ with $h(0)>0$ is
the p.g.f. of an infinitely divisible distribution over the  on-negative integers if
and only if $h(z)=\exp\{\int_{z}^1 H(s) ds  \}$ for an absolutely monotone function
$H$ (cf. Thm 4.2 \cite{SteutelVanHarn}).
Finally, note that  $m_0^{-1} \partial_z \hat g_\infty(z)=z^{-1} \hat{g_{\infty}^*}(z)$
with $ \hat{g^*_{\infty}}(z)=\sum_{k \geq 0} z^k g^*_{\infty}(k)$ and $g^*_{\infty}(k)=k g_{\infty}(k)/m_0$, that is the size-biased version of $g_{\infty}$.
At this stage one can re-write \eqref{stat_grazing} as
\[
 z R(z) \hat g_\infty(z) =   \hat{g^*_\infty}(z).
\]
Therefore, in terms of random variables, \eqref{stat_grazing} takes the form
\begin{equation}\label{sizebiasedinfdiv1}
 V_\infty^* \stackrel{d}{=}  V_\infty + 1+ \zeta+\sum_{k=1}^{\gamma} \xi_k,
\end{equation}
where $V_\infty$ is the size biased version of $V_\infty$.
\end{proof}

Adapting Lemma 2.2 in \cite{CaceresToscani2007} one gets a further result.

\begin{lemma}
Let $(\tilde f_n)_n,(\tilde g_n)_n,\tilde f,\tilde g$ be the Laplace transform of
probability measures on the positive real axis, such that $\tilde  f_n \to \tilde f$,
$\tilde  g_n \to \tilde g$, $M_1(f_n) = M_1(g_n)=m_0$, and for some $\eps>0$
\begin{equation}\label{cond1}
\sup_n \Big \{ \max  \Big
\{\int_{\RE^+} v^{r+\eps} df_n,\int_{\RE^+} v^{r+\eps} dg_n
 \Big\}
  \Big\}=M_\eps<+\infty.
\end{equation}
Then
\begin{equation}\label{lemmacacerestoscani}
d_r^* ( f, g)\leq \liminf_n d_r^*( f_n, g_n).
\end{equation}
\end{lemma}

\begin{proof}
Recall that pointwise convergence of Laplace transforms yields uniform convergence on
every compact set, hence, for every $0<\delta<R<+\infty$,
\begin{equation}\label{unifconvn}
C_n=C_n(\delta,\eps):=\sup_{\xi \in [\delta,R]}|\tilde f_n(\xi)-\tilde f(\xi)|+
\sup_{\xi \in [\delta,R]} |\tilde g_n(\xi)-\tilde g(\xi)| \to 0.
\end{equation}
Moreover, since pointwise convergence of Laplace transform yields weak convergence of
distribution (cf. Theorem 25.11 and Corollary 25.12 in \cite{Billingsley}), using
condition \eqref{cond1} one obtains $M_1(f) = M_1(g)=m_0 $ and
\[
\max\Big \{\int v^{r+\eps} f(dv),\int v^{r+\eps}  g(dv)\Big \}\leq M_\eps.
\]
Using again \eqref{cond1},  \eqref{D4} yields
\[
|\tilde f_n(\xi)-\tilde f(\xi)|\xi^{-r}
\leq |\tilde f_n(\xi)-1-m_0\xi|\xi^{-r}  +|\tilde f(\xi)-1-m_0\xi|\xi^{-r}
\leq 2 c_{r+\eps} \xi^\eps M_\eps.
\]
Hence
\[
\sup_{\xi \in (0,\delta)} |\tilde f_n(\xi)-f(\xi)|\xi^{-r} \leq
2 c_{r+\eps} \delta^\eps M_\eps.
\]
Now, consider that 
\[
\begin{split}
d_r^* (f_n,f) & \leq \max\Big \{\sup_{\xi \in (0,\delta)} |\tilde f_n(\xi)-f(\xi)|\xi^{-r}  ,\sup_{\xi  \in  [\delta,R)} |\tilde f_n(\xi)-f(\xi)|\xi^{-r} ,
\sup_{\xi \geq R} |\tilde f_n(\xi)-f(\xi)|\xi^{-r}
    \Big \} \\
    &
\leq \max\Big \{2 c_{r+\eps} \delta^\eps M_\eps,  \,\, C_n, \,\, \frac{2}{R^r} \Big \},
\\
\end{split}
\]
Since the same estimate holds for $d_r^* (g_n,g)$,
\[
\lim_n [d_r^* (f_n,f)+d_r^* (g_n,g)]=0.
\]
To conclude it suffices to observe that
\[
d_r^* (f,g) \leq
d_r^* (f_n,g_n) +d_r^* (f_n,f)+d_r^* (g_n,g).
\]

\end{proof}


\begin{theorem} Let $X$ and $Y$ be defined as in  \eqref{variablegrazing}, with $b_1=b_2$, and
$\tilde X$ and $\tilde Y$ satisfying \eqref{condizionemedia}-\eqref{condiziomomentosecondo}.
Let $f_{t,\eps}^{(i)}$, $i=1,2$, be two solutions of \eqref{boltz}
with initial conditions  $f_0^{(i)}$ with mean equal to $m_0$ and finite second moment.
 Let $g_t^{(1)}$ and $g_t^{(2)}$ be the corresponding quasi invariant collision limits
of Proposition \ref{prop:limite}.
Then, for every $1<r<2$,
\[
 d_{r}^*(g_t^{(1)},g_t^{(2)}) \leq  d_{r}^* (g_0^{(1)},g_0^{(2)}) e^{-r (1-\laa \tilde X \raa )^{r-1} t}.
\]
 \end{theorem}

\begin{proof}
We first observe that $\laa \tilde X \raa+ \laa \tilde Y \raa=1$ implies $1-\laa \tilde X \raa >0$.
Explicit computations and Taylor formula give\begin{equation}
\begin{split}
 \a_r& = \la X\ra^r+\la Y \ra^r-1=[1+\eps (\laa \tilde X\raa-1) ]^r +\eps^r \laa \tilde X\raa^r -1 \\
  &=r \eps (\laa \tilde X \raa-1)^{r-1} + \frac{r(r-1)}{2} \eps^2 (\laa \tilde X\raa-1)^2 (1+Q_\eps)^{r-2}+\eps^r \laa \tilde X\raa^r, \\
\end{split}
\end{equation}
with $|Q_\eps|<\eps (1-\laa \tilde X\raa)$. Therefore
\begin{equation}\label{areps}
a_r= r \eps(1-\laa \tilde X \raa )^r +o(\eps).
\end{equation}
Combining \eqref{vart}, \eqref{xyeps2} and \eqref{areps} it follows
\[
M_2(f_{t/\eps,\eps}^{(i)}) \leq C[1+e^{-\a_2(\eps)t/\eps}]\leq C',
\]
for every $t>0$ and $\eps>0$.
Since $\tilde  g_t^{(i)}(z)= \lim_{\eps \to 0} \tilde f_{t/\eps,\eps}^{(i)}(z)$,
combining Theorem \ref{thmconvergence}, \eqref{lemmacacerestoscani} and \eqref{areps}
\[
 d_r^*(g_t^{(1)},g_t^{(2)}) \leq \liminf_\eps d_r^*(f_{t/\eps}^{(1)},f_{t/\eps}^{(2)}) \leq d_r^* (g_0^{(1)},g_0^{(2)}) e^{- r (1-\laa \tilde X \raa )^{r-1}t} .
\]

\end{proof}

\vfill\eject
\appendix
\section{}

\subsection{Recurrence relations}

We state without proof two  previously used results.

\begin{lemma}\label{LemmaGamma} Let $0<a_0<+\infty$, and, for every $n\ge 1$, let $a_n \geq 0 $ satisfy, for some $0<\Delta<+\infty$, the inequality
\begin{equation}\label{recuran}
a_n \leq \frac{\Delta}{n} \sum_{j=0}^{n-1} a_j.
\end{equation}
 Then, for every $n\ge 1$
\begin{equation}\label{thesisan}
a_n \leq a_0 \frac{\Gamma(\Delta+n)}{\Gamma(\Delta)\Gamma(n+1)}.
\end{equation}
Moreover the equality sign in \eqref{recuran} implies  the equality sign in
\eqref{thesisan}.
\end{lemma}

\begin{lemma}\label{LemmaGamma2} Let $0<a_0<+\infty$, and, for every $n\ge 1$, let $a_n \geq 0 $ satisfy the inequality
\begin{equation*}\label{recuranB}
a_n \leq \frac{A}{n} \sum_{j=0}^{n-1} a_j +B n^\beta,
\end{equation*}
for some $0<A,B,\beta<+\infty$. Then there is $K=K(A,B,\beta,a_0)\geq 1$ such that
for every $n \geq 0$
\begin{equation*}\label{thesisanB}
a_n \leq K^{n+1}.
\end{equation*}
\end{lemma}

We conclude this short section by recalling that, for any $t>0$ and $\Delta>0$  \cite{Gradshteyn}
\begin{equation}\label{serieHypergeo}
e^{(\Delta-1)t} = \sum_{n \geq 0} e^{-t}(1-e^{-t})^n\frac{\Gamma(\Delta+n)}{\Gamma(\Delta)\Gamma(n+1)}.
\end{equation}

\subsection{Inequalities for Laplace transforms and cumulants}
Given a random variable $Z$ taking values in $\RE^+$ let us set, for $\xi >0$
\[
L_Z(\xi)=\E[e^{-\xi Z}] \quad  \text{and} \quad k_Z(\xi)=\log(L_Z(\xi)).
\]
Since for all $x>0$
\[
|e^{-x}-1|\leq |x|,
\]
and for every $r \in (1,2)$ one can find a constant $c_r$ such that
\begin{equation}\label{D3}
|e^{-x}-1+x|\leq c_r |x|^r,
\end{equation}
whenever $ \E[Z^r]<+\infty$ and  $\xi>0$ it follows
\begin{equation}\label{D4}
|L_Z(\xi)-1-\xi \E[Z]| \leq c_r \xi^r \E[Z^r],
\end{equation}
and,  if $ \E[Z]<+\infty$,
\begin{equation}\label{D5}
|L_Z(\xi)-1| \leq \xi\E[Z]
\end{equation}
The following lemmas contain useful bounds for cumulants.

\begin{lemma} Let $ \E[Z^r]<+\infty$ for some $r \in (1,2)$. Then, if $0\leq \xi\leq (2\E[Z])^{-1}$,
\begin{equation}\label{D6}
|k_Z(\xi) + \xi \E[Z]| \leq c_r \xi^r \E[Z^r]+\xi^2 (\E[Z])^2.
\end{equation}
Moreover
\begin{equation}\label{ineqKX}
 \sup_{\xi>0} \frac{|k_Z(\xi)|^{r}}{|\xi|^{r}}=\E[ Z ]^{r}.
\end{equation}
\end{lemma}

\begin{proof}
For every complex number $z$ write
\[
\log(1+z) = z(1+ R(z)).
\]
Then $|R(z)|\leq |z|$ for every $z$ such that $|z|\leq 1/2$ (cf. Proposition 8.46 in \cite{Breiman}).
Thanks to \eqref{D5},  $0\leq \xi\leq (2\E[Z])^{-1}$ implies
\[
|L_Z(\xi)-1| \leq |\xi| \E[Z] \leq \frac{1}{2}.
\]
By virtue of \eqref{D4} and \eqref{D5}
\[
\begin{split}
|k_Z(\xi)+\xi \E[Z]|& =|\log(1+L_Z(\xi)-1)+\xi \E[Z]| =|(L_Z(\xi)-1)(1+R(L_Z(\xi)-1))
+ \xi \E[Z]|\\
&
\leq
|L_Z(\xi)-1+ \xi \E[Z]|+[L_Z(\xi)-1]^2
\leq x_r \xi^r \E[Z^r]+\xi^2 (\E[Z])^2.
\\
\end{split}
\]
To obtain \eqref{ineqKX}, note that
$\xi \mapsto L_Z(\xi)=\E[e^{-\xi Z}]$ is a non-increasing function and
$\E[ e^{-\xi Z}] \leq 1$. Hence $\xi \mapsto |\log(\E[ e^{-\xi Z}])|=|k_X(\xi)|$ is non-increasing. So
\[
\sup_{\xi>0} \frac{|k_X(\xi)|^{r}}{|\xi|^{r}} =\lim_{\xi \to 0}  \frac{|k_X(\xi)|^{r}}{|\xi|^{r}}.
\]
Thanks to the previous estimate of the Lemma
\[
\frac{|k_Z(\xi)|^{r}}{|\xi|^{r}}=\frac{|\xi \E[Z]+\rho(\xi)|^r}{|\xi|^r},
\]
with $|\rho(\xi)|^r \leq C |\xi|^{2r}$. This gives \eqref{ineqKX}.
\end{proof}

\begin{lemma} Let $Z$ be a random variables taking values in
$\N$ and $X$ a non-negative real random variable. Assume that $ \E[X^r]+\E[Z^r]<+\infty$ for some $r \in (1,2)$. Then there is a positive constant $C=C(X)$ such that, for all $\xi>0$
\begin{equation}\label{D7}
|L_Z(-k_X(\xi))-L_Z(\xi \E[X])| \xi^{-r} \leq C(X)[1+\E[Z^r].
\end{equation}
\end{lemma}

\begin{proof}
Let us define
\[
C(X)=\max\{c_r (\E[X^r]+ 2 \E[X]^r)+2^{r-2}\E[X]^r,2^{r+1}\E[X]^r\}.
\]
If $\xi \geq 1/(2\E[X])$, it holds
\[
|L_Z(-k_X(\xi))-L_Z(\xi \E[X])| \xi^{-r}  \leq 2^{r+1}\E[X]^r \leq C(X).
\]
If now $\xi <  1/(2\E[X])$ one obtains
\[
\begin{split}
|L_Z(-k_X(\xi))- & L_Z(\xi \E[X])| \xi^{-r}   \leq
|L_Z(-k_X(\xi))-1-k_X(\xi) \E[Z]| \xi^{-r}\\
& +|1-\xi \E[X] \E[Z]-L_Z(\E[X]\xi)| \xi^{-r}
+ E[Z]|\xi \E[X]+k_X(\xi)|\xi^{-r}
\\
& \leq c_r\E[Z^r]\{ \xi^{-r} |k_X(\xi)|^r+ \E[X]^r \}
+ \E[Z^r]\{ c_r \E[X^r]+\xi^{2-r}\E[X]^2\}
\\
& \leq \E[Z^r] [c_r (\E[X^r]+ 2 \E[X]^r)+2^{r-2}\E[X]^r ]\\
\end{split}
\]
Note that the last two inequalities from below have been obtained by using \eqref{D4} and \eqref{D6} and \eqref{ineqKX}, respectively.
\end{proof}

\begin{lemma}\label{lemma3}
Let
\(
C:=\sup_{\xi>0}  {|1-e^{-\xi}|}{|\xi|^{-1}}< +\infty,
\)
and $B=\sup_{0<\xi\leq \log(2)}  {|\xi|}{|1-e^{-\xi}|^{-1}}$. Then, for every $r \geq 1$
\[
 d^*_{r}(f,g) \leq C^r d_{r}(f,g),
\]
and
\[
d_{r}(f,g) \leq \max \{ 2^{r+1}, B^r d_{r}^*(f,g) \}.
\]
Moreover, if $M_1(f)=M_1(g)$, and $M_r(f)$ and $M_r(g)$ are finite for some $r \in (1,2]$, then
\[
 d^*_{r}(f,g) \leq c_r[  M_r(f)+M_r(g)],
\]
for a suitable constant $c_r$.
\end{lemma}

\begin{proof}
Noticing that $\hat f(e^{-\xi})=\tilde f(\xi)$, one can write
\[
\frac{|\tilde f(\xi)-\tilde g(\xi)|}{|\xi|^{r}}
=\frac{|\hat f(e^{-\xi})-\hat g(e^{-\xi}) |}{|1-e^{-\xi}|^r} \frac{|1-e^{-\xi}|^r}{|\xi|^{r}}.
\]
Hence
\(
 d^*_{r}(f,g) \leq C^r d_{r}(f,g)
\).
Moreover,
\[
\begin{split}
 d_{r}(f,g)& =\max \left \{ \sup_{s \in (0,1/2)}
 \frac{|\hat f(s)-\hat g(s) |}{|1-s|^r}, \sup_{s \in (1/2,1)}
 \frac{|\hat f(s)-\hat g(s) |}{|1-s|^r} \right\} \\
 &
 \leq
 \max \left \{ 2^{r+1}, \sup_{\xi \in (0,\log(2))}
  \frac{|\hat f(e^{-\xi})-\hat g(e^{-\xi}) |}{|\xi|^r}
  \frac{|\xi|^{r}}{|1-e^{-\xi}|^r}
 \right \}
 \\
  &
 \leq
 \max \left \{ 2^{r+1}, B^r d_{r}^*(f,g)
 \right \} .
 \\
 \end{split}
\]
Thanks to
\eqref{D3}
\[
\begin{split}
\frac{|\tilde f(\xi)-\tilde g(\xi)|}{|\xi|^{r}} &=
\frac{|\int (e^{-\xi v}-1+v\xi) f(dv)-\int (e^{-\xi v}-1+v\xi) g(dv)|}{|\xi|^{r}}
\\
&
\leq c_r [  M_r(f)+M_r(g)] .
\\
\end{split}
\]
\end{proof}

\subsection{Moments of random sums}

Let $V_1$ and $V_2$ two integer valued random variables.
Let $(X_k)_{k \geq 1}$ and $(Y_k)_{k \geq 1}$
two sequences of iid non-negative  random variables and assume that
$V_1,V_2$,\\$(X_k)_{k \geq 1}$,$(Y_k)_{k \geq 1}$ are independent.

\begin{lemma}\label{inequalitysum}
Under the previous assumptions let $r \in (1,2]$ and
assume that $\la Y^r +Y^r \ra<+\infty$ and $\la V_1^r+V_2^r \ra<+\infty$.
Then, if $r=2$
\begin{equation}\label{squarethesum}
\begin{split}
\la \Big (\sum_{i=1}^{V_1}
X_i +\sum_{i=1}^{V_2} Y_i  \Big )^2\ra  & =
\la V_1 \ra  Var(X_1)+\la V_2 \ra Var(Y_1) \\
&
+2  \la V_1 \ra \la V_2 \ra  \la X_1 \ra \la Y_2 \ra +
\la V_1^2  \ra \la X_1 \ra^2+\la V_2^2  \ra\la Y_1 \ra^2,
\\
\end{split}
\end{equation}
while if $r \in (1,2)$
\begin{equation}\label{inequalityrmoment}
\begin{split}
 \la \Big (    \sum_{i=1}^{V_1} X_i
+\sum_{i=1}^{V_2}
Y_i \Big )^r  \ra
&  \leq \la V_1 \ra\la  X_1^r \ra  + \la V_2 \ra\la  Y_1^r \ra +\\
  &
     \Big (2 \la V_1 \ra \la V_2 \ra  \la X_1 \ra \la Y_1\ra \Big)^{\frac{r}{2}}
+  \la V_1^r\ra  \la X_1 \ra^r  + \la V_2^r \ra \la Y_2 \ra^r.
\\
\end{split}
\end{equation}
\end{lemma}

\begin{proof} 
We only proof \eqref{inequalityrmoment}.
Let $n_1$ and $n_2$ two integer numbers. Since $r/2<1$, for any given set of positive numbers $z_1,\dots,z_n$,
$(\sum_{i=1}^n z_i \big)^{r/2} \leq \sum_{i=1}^n z_i^{r/2}$. This implies
\[
\begin{split}
&\la \Big (  \sum_{i=1}^{n_1} X_i+\sum_{i=1}^{n_2} Y_i   \Big )^{2 \frac{r}{2}}   \ra
=
\la \Big (
\Big ( \sum_{i=1}^{n_1} X_i \Big )^2   + \Big (\sum_{i=1}^{n_2} Y_i \Big )^2
+  2 \sum_{i=1}^{n_1} X_i \sum_{i=1}^{n_2} Y_i  \Big )^{\frac{r}{2}} \ra
\\
& =
\la \Big (
\sum_{i=1}^{n_1} X_i^2 + \sum_{i=1}^{n_2} Y_i^2 +  2\sum_{i=1}^{n_1} X_i \sum_{i=1}^{n_2} Y_i
+ \sum_{i\not =j,i,j=1 }^{n_1} X_i X_j+ \sum_{i\not =j,i,j=1 }^{n_2} Y_i Y_j
 \Big )^{\frac{r}{2}}
\ra
\\
&
\leq
\la
\sum_{i=1}^{n_1} X_i^r + \sum_{i=1}^{n_2} Y_i^r +  2^{\frac{r}{2}} \Big (\sum_{i=1}^{n_1}X_i \sum_{i=1}^{n_2}Y_i \Big)^{\frac{r}{2}} 
+  \Big ( \sum_{i\not =j,i,j=1 }^{n_1} X_i X_j\Big)^{\frac{r}{2}}  +
 \Big (\sum_{i\not =j,i,j=1 }^{n_2} Y_i Y_j\Big)^{\frac{r}{2}}
\ra
\\
&
\leq n_1 \la  X_1^r \ra +n_2 \la  Y_1^r \ra +
 2^{\frac{r}{2}}  \la \sum_{i=1}^{n_1} X_i \sum_{i=1}^{n_2} Y_i \ra^{\frac{r}{2}}
+  \la \sum_{i\not =j,i,j=1 }^{n_1} X_i X_j\ra^{\frac{r}{2}}  +
\la \sum_{i\not =j,i,j=1 }^{n_2} Y_i Y_j\ra^{\frac{r}{2}}
\\
& \leq n_1 \la  X_1^r \ra +n_2 \la  Y_1^r \ra +
 2^{\frac{r}{2}}  \Big (  n_1 n_2 \la X_1\ra \la Y_1\ra \Big)^{\frac{r}{2}}
+  n_1^r \la X_1 \ra^r  + n_2^r  \la Y_1\ra^r.
\\
\end{split}
\]
Note that we used Jensen's inequality to get the last line. To conclude let us apply the previous inequality conditionally on $V_1$ and $V_2$ to get
\[
\begin{split}
 &
 \la \Big (    \sum_{i=1}^{V_1} X_i
+\sum_{i=1}^{V_2}
Y_i \Big )^r  \ra
  \leq \la V_1 \ra\la  X_1^r \ra  + \la V_2 \ra\la  Y_1^r \ra +
       \la V_1^{\frac{r}{2}}  \ra \la V_2^{\frac{r}{2}}
      \ra   \Big (2 \la X_1 \ra \la Y_1\ra \Big)^{\frac{r}{2}}
+ \\ & \qquad\qquad\qquad\qquad\qquad\quad\la V_1^r\ra  \la X_1 \ra^r  + \la V_2^r \ra \la Y_2 \ra^r.
\\
\end{split}
\]
and then apply the Jensen's inequality one more time to get
$ \la V_i^{\frac{r}{2}}  \ra \leq \la V_i \ra^{\frac{r}{2}}$.
\end{proof}

\subsection{Gronwall's inequality}
A well-known version of the Gronwall lemma is the following.

\begin{lemma}\label{gronwall1}  Suppose that $t \mapsto u(t)$ and $t \mapsto h(t)$
are integrable functions on the interval $[0,T]$. If for some constant $C>0$
\[
u(t) \leq h(t) + C \int_0^t u(s)ds \qquad 0 \leq t \leq T,
\]
then, for every $t \in [0,T]$
\[
u(t) \leq h(t)+ C \int_0^t e^{C(t-s)}h(s)ds.
\]

\end{lemma}

A simple consequence of the previous Lemma is the following.

\begin{lemma}\label{gronwall2} Let $t \mapsto u(t)$ be a
positive measurable function on the interval $[0,+\infty)$
such that:
\begin{itemize}
\item[(i)]
$I=\{t : u(t)<+\infty\}$ is an open interval $(0,s_0)$,
\item[(ii)] $u$ is continuous on $I$
\item[(iii)] if  $s_0<+\infty$ then
\[
\lim_{t \to s_0^-}u(s)=+\infty.
\]
\end{itemize}
If there is a  non-negative continuous function $t \mapsto h(t)$  on
 $[0,+\infty)$ and a constant $C>0$ such that
\[
u(t) \leq h(t) + C \int_0^t u(s)ds,
\]
for every $t\geq 0$, then
\[
u(t) \leq h(t)+ C \int_0^t e^{C(t-s)}h(s)ds,
\]
for every $t >0$ and, in particular, $s_0=\infty$. Finally,
if $h(t)$ is of bounded variation then
\[
u(t) \leq e^{Ct}h(0)+\int_0^t e^{C(t-s)}dh(s).
\]
\end{lemma}

\begin{proof} It is enough to prove that
$s_0=+\infty$. Indeed in this case the thesis follows from the previous Lemma.
Assume that $s_0<+\infty$. Let $H(t)=h(t)+ C \int_0^t e^{C(t-s)}h(s)ds$.
By (ii) and the fact that $H(t)$ is bounded on every finite interval,
 one can choose $\eps$ in such a way that
$u(s_0-\eps) > H(s_0-\eps)$. Now
$u$ is continuous in $(0,s_0)$ and hence it is bounded
on $(0,s_0-\eps)$. Hence, by the previous Lemma,
$u(s_0-\eps)\leq H(s_0-\eps)$ which gives a contradiction.
The last part of the proof follows resorting to integration by parts.
\end{proof}

\vskip 1cm
\noindent
{\textsc{Acknowledgements:}} This work has been written within the activities of GNFM  (G.T.) and GNAMPA (F.B.) groups of INdAM (National Institute of High Mathematics), and partially supported by  MIUR project ``Optimal mass transportation, geometrical and functional inequalities with applications'' (G.T.). 

\end{document}